\documentclass[a4paper,11pt,fullpage]{article}
\usepackage{amsthm}
\usepackage{amsfonts}
\usepackage{xr}
\usepackage{graphicx}
\usepackage{amsmath}
\usepackage{epsfig,subfigure}
\usepackage{color}
\usepackage{hyperref} 

\def\Z{\mathbb{Z}}
\def\R{\mathbb{R}}
\def\N{\mathbb{N}}

\def\epsilon{\varepsilon}
\def\hat{\widehat}
\def\tilde{\widetilde}

\newcommand{\me}{\mathrm{e}}

\newcommand{\SE}{\setcounter{equation}{0} \section}
\newcommand{\be}{\begin{equation}}
	\newcommand{\ee}{\end{equation}}
\newcommand{\baa}{\begin{array}}
	\newcommand{\eaa}{\end{array}}
\newcommand{\ba}{\begin{eqnarray}}
	\newcommand{\ea}{\end{eqnarray}}

\newtheorem{theo}{\bf Theorem}[section]
\newtheorem{lem}[theo]{\bf Lemma}
\newtheorem{pro}[theo]{\bf Proposition}

\newtheorem{cor}[theo]{\bf Corollary}
\newtheorem{defi}[theo]{\bf Definition}
\newtheorem{rem}[theo]{\bf Remark}

\numberwithin{equation}{section}

\begin{document}
\date{}
\title{\bf{ Average speeds of time almost periodic traveling waves for rapidly/slowly oscillating reaction-diffusion equations}}
\author{Weiwei Ding\footnote{School of Mathematical Sciences, South China Normal University, Guangzhou 510631, China. dingweiwei@m.scnu.edu.cn}}
\maketitle
\begin{abstract}
This paper is concerned with the propagation dynamics of time almost periodic reaction-diffusion equations. Assuming the existence of a time almost periodic traveling wave connecting two stable steady states, we focus especially on the asymptotic behavior of average wave speeds in both rapidly oscillating and slowly oscillating environments. We prove that, in the rapidly oscillating case, the average speed converges to the constant wave speed of the homogenized equation; while in the slowly oscillating case, it approximates the arithmetic mean of the constant wave speeds for a family of equations with frozen coefficients. In both cases, we provide estimates on the convergence rates showing that, in comparison to the limiting speeds, the deviations of average speeds for almost periodic traveling waves are at most linear in certain sense.  Furthermore, our explicit formulas for the limiting speeds indicate that temporal variations have significant influences on wave propagation. Even in periodic environments, it can alter the propagation direction of bistable equations.
\end{abstract}

%%%%%%%%%%%%%%%%%%%%%%%%%%%%%%%%%%%%%%%%
%%%%%%%%%%%%%%%%%%%%%%%%%%%%%%%%%%%%%%%%

\SE{Introduction and main results}
\subsection{Introduction}
In this paper, we study the one-dimensional reaction-diffusion equation of type
\begin{equation}\label{eq-main}
u_t=u_{xx}+f\left(t/T,u\right),\ \ t\in\R,\ x\in\R,
\end{equation}
with parameter $T>0$. The function $f:\R\times \R \to\R,\ (t,u)\mapsto f(t,u)$ is of class $C^2$, and $f$, $\partial_u f$ and $\partial_{uu} f$ are bounded in $\R\times E$ for any bounded set $E\subset \R$. These assumptions of $f$ as well as the following two will be understood throughout the paper without further mention.

\begin{itemize}
\item[(A1)] $f(t,u)$ is almost periodic in $t$ uniformly with respect to $u$ in bounded sets (we will recall the definition of almost periodic function and some basic properties in Section 2.1).
\end{itemize}

\begin{itemize}
\item[(A2)] $f(t,0)=f(t,1)=0$ for all $t\in\R$, and $0$ and $1$ are uniformly (in $t$) stable zeroes of $f(t,\cdot)$, in the sense that there exist $\gamma_0>0$ and $\delta_0\in(0,1/2)$ such that
\begin{equation}\label{asspars}\left\{\baa{ll}
f(t,u)\le-\gamma_0 u & \hbox{for all }\,(t,u)\in\R\times[-\delta_0,\delta_0],\vspace{5pt}\\
f(t,u)\ge\gamma_0(1-u) & \hbox{for all }\,(t,u)\in\R\times[1-\delta_0,1+\delta_0].\eaa\right.
\end{equation}
\end{itemize}

The assumption (A2) implies in particular that $\max\{\partial_u f(t,0), \partial_u f(t,0) \}\leq -\gamma_0$ for all $t\in\R$, and that $u^+=1$ and $u^-=0$ are two stable solutions of the corresponding ODE
\begin{equation}\label{ode-f}
	\frac{du}{dt}=f\left(t/T,u\right)
\end{equation}
in the sense that $u(t;u_0)- u^{\pm}\to 0$ as $t\to+\infty$ uniformly for $u_0\in \R$ with $|u_0-u^{\pm}|\leq \delta_0$, where $u(t;u_0)$ is the solution of \eqref{ode-f} with initial value $u_0$.

Equation \eqref{eq-main} arises in modeling a variety of physical and biological phenomena, such as population dynamics, gene developments, phase transitions, see e.g., \cite{abc,aw,fm,shen1,sk,xin00}. The time almost periodic dependence of the forcing term $f(t/T,u)$ with respect to $t$ allows one to take into account general temporal oscillations in natural environments, where $1/T$ characterizes the frequency of these oscillations. The main objective of this paper is to establish the asymptotic limits as $T\to 0^+$ and as $T\to +\infty$ of some fundamental spreading quantities associated with almost periodic traveling wave solutions of \eqref{eq-main}.

Traveling wave is of great importance to understand the evolution process in reaction-diffusion equations, as it is the key notion for the description of the invasion of a state by another one (here, the states are the constants $0$ and $1$).
\begin{defi}\label{de-al-tw}
By {\it an almost periodic traveling wave solution} connecting $0$ and $1$ for \eqref{eq-main}, we mean an entire solution $u_T:\R\times\R\to [0,1]$ such that there are $U_T\in C^{1,2}_{t,x}(\R\times \R)$ and $c_T\in C^1(\R)$ such that
\begin{equation*}\left\{\baa{ll}
	u_T(t,x)=U_T(t,x-c_T(t))& \hbox{for }\, t\in\R,\, x\in\R, \vspace{5pt}\\
	U_T(t,-\infty)=1,\,\, U_T(t,+\infty)=0 & \hbox{uniformly for }\, t\in\R,\eaa\right.
\end{equation*}
and that $U_T(t,x)$ (called the {\it wave profile}) is almost periodic in $t$ locally uniformly with respect to $x$, and $c_T'(t)$ (called the {\it wave speed}) is almost periodic in $t$. 
% If $c_T(t)$ is a bounded function, then $u_T(t,x)$ is called an almost periodic standing wave solution.
\end{defi}

The notion of almost periodic traveling wave was introduced by Shen \cite{shen1,shen2}. It is a natural extension of a more classical notion of time periodic traveling wave in the case where $f(t,u)$ depends periodically on $t$. More precisely, {\it a time periodic traveling wave solution} of \eqref{eq-main} is defined to be a solution of form $u_T(t,x)=V_T(t,x-d_Tt)$, where the wave profile $V_T(t,x)$ changes periodically in $t$ with period $lT$ and the speed $d_T$ is a constant (see \cite{abc}), provided that $f(t,u)$ is $l$-periodic in $t$. It is well known that (see e.g., \cite{shen1,shen4}), in the periodic case, if $u_T(t,x)=U_T(t,x-c_T(t))$ is an almost periodic traveling wave of \eqref{eq-main}, then $c'_T(\cdot)$ and $U_T(t,x)$ are necessarily periodic in $t$ with the same period $lT$, and hence, $u_T(t,x)=V_T(t,x-d_Tt)$ is a periodic traveling wave, where $d_T=1/(lT)\int_0^{lT}c'_T(t)dt$ and $V_T(t,x)=U_T(t,x-c_T(t)+d_Tt)$.
Clearly, when $f$ is independent of $t$, both almost periodic traveling wave and periodic traveling wave simplify to the classical notion of traveling wave, that is a solution takes the form $u(t,x)=\phi(x-ct)$ for some constant $c$ and some one variable function $\phi(x)$. It worth mentioning that the constant speed $c$ has the same sign as that of $\int_0^1 f(u)du$, while in the periodic case, determining the sign of wave speed $d_T$ remains an open problem.

An almost periodic traveling wave solution of \eqref{eq-main} was proved to exist in the case where $f$ is of bistable type in the sense that, in addition to the conditions (A1)-(A2), \eqref{ode-f} is assumed to have a unique almost periodic solution lying strictly between $0$ and $1$, and it is unstable (see \cite{shen2,shen4}). Note that almost periodic traveling wave solutions may exist when $f$ is of {\it multistable type} in the sense that \eqref{ode-f} may have other stable almost periodic solutions lying between $0$ an $1$. For such existence results, we refer to \cite{fm,dll} in the autonomous case (see also Proposition \ref{multi-existence} below) and \cite{dm} in the periodic case. In this paper, we assume the existence in almost periodic case, as stated below.

\begin{itemize}
\item[(A3)]
For each $T>0$, problem \eqref{eq-main} admits an almost periodic traveling wave solution $u_T(t,x)=U_T(t,x-c_T(t))$ connecting $0$ and $1$.
\end{itemize}

Under the above assumptions, it has been proved in \cite{shen1,shen3} that the wave profile $U_T(t,x)$ is unique up to spatial shifts by $C^1$ time almost periodic functions, decreasing in $x$ and the traveling wave solution $u_T(t,x)$ is globally stable. As these results will be frequently used in the present work, for clarity, we summarize them in the following theorem.

\begin{theo}\label{th-known} {\rm (\cite{shen1,shen3})}
Let {\rm (A1)-(A3)} hold, and for each $T>0$, let $u_T(t,x)=U_T(t,x-c_T(t))$ be an almost periodic traveling wave solution of \eqref{eq-main} connecting $0$ and $1$. The following statements hold true:
\begin{itemize}
	
\item [{\rm (i)}] (Monotonicity) $\partial_x U_T(t,x)<0$ for all $t\in\R$, $x\in\R$;

\item [{\rm (ii)}] (Stability) There exists $\delta'\in (0,\delta_0)$ such that for any solution $u(t,x)$ of
\eqref{eq-main}  satisfying
\begin{equation}\label{initial-s}
	U_T(\tau,x-c_T(\tau)-\zeta^1)-\delta' \leq  u(\tau,x) \leq U_T(\tau,x-c_T(\tau)-\zeta^2)+\delta' \, \hbox{ for all }\, x\in\R,
\end{equation}
for some $\tau \in\R$ and $\zeta^1\leq \zeta^2$ in $\R$, there holds
\begin{equation}\label{stability}
	\sup_{x\in\R} |u(t,x)-U_T(t,x-c_T(t)-\xi^*_T)|\to 0 \,\, \hbox{ as }\,\,t\to+\infty,
\end{equation}
for some constant $\xi^*_T\in\R$ depending on $T$;

\item [{\rm (iii)}] (Uniqueness) If \eqref{eq-main} admits another almost periodic traveling wave solution $\tilde{u}_T(t,x)=\tilde{U}_T(t,x-\tilde{c}_T(t))$ connecting $0$ and $1$, then there exists a $C^1$ almost periodic function $t\to\zeta_T(t)$ such that
$$U_T(t,x)=\tilde{U}_T(t,x-\zeta_T(t)) \quad \hbox{and}\quad  c_T'(t)=\tilde{c}_T'(t)-\zeta_T'(t) $$
for $t\in\R$, $x\in\R$.
\end{itemize}
\end{theo}

More precisely, we refer to \cite[Theorem 4.1]{shen1} and \cite[Theorem 6.1 (2)]{shen3} for the monotonicity\footnote{Although the main results in \cite{shen1} were proved under the assumption that $f$ is of  bistable type, one can check that the sliding arguments used to prove the monotonicity can be easily extended to the multistable case with some trivial modifications.} and the uniqueness of time almost periodic traveling wave solutions, respectively. The stability follows directly from \cite[Theorem 6.1 (1)]{shen3} and its proof\footnote{\cite[Theorem 6.1]{shen3} stated the local stability of time almost periodic traveling wave solutions (i.e., \eqref{initial-s} is assumed to hold with $\zeta^1=\zeta^2$),  but the squeezing arguments used in the proof actually showed the global stability as stated above.}. We point out that, apart from the aforementioned works, time almost periodic traveling waves, as well as more generalized transition waves in the sense of \cite{bh12}, were extensively investigated for equation \eqref{eq-main} with monostable and combustion nonlinearities, see e.g., \cite{hs,nr1,shen5,shen6,ss} and references therein. We also refer to \cite{m,nr2,lmn,clou,liang} towards the study of almost periodic traveling waves in spatially heterogeneous media.

Let us now introduce an important notion concerned with the propagating of time almost periodic traveling waves.
For each $T>0$, denote by $\bar{c}_T$ the {\it average speed} of $c'_T(t)$, that is,
\begin{equation}\label{mean-speed}
	\bar{c}_T=\lim_{t\to+\infty} \frac{1}{t}\int_s^{t+s} c'_T(\tau)d\tau,
\end{equation}
where the convergence is understood to be uniform in $s\in\R$. Since $c'_T(\cdot)$ is an almost periodic function, it is clear that the limit $\bar{c}_T$ exists independently of $s$. It is also easily seen that, in the periodic case, $\bar{c}_T$ is the constant wave speed of time periodic traveling wave. Furthermore, by Theorem 1.1 (iii), $\bar{c}_T$ is the unique average speed of equation \eqref{eq-main}.
This leads to an interesting question: how the average speed $\bar{c}_T$ depends on the oscillation parameter $T$ ? This problem is important in many physical and biological phenomena, but very little is known so far, even in the periodic case. To study this problem, let us first present some trivial observations when $f$ is of Fisher-KPP type or of a special bistable type. 

Firstly, we consider the Fisher-KPP case, that is, instead of (A2)-(A3), $f(t,u)$ is assumed to satisfy
$f(\cdot,0)=f(\cdot,1)=0$ in $\R$ and
\begin{equation}\label{fisher-kpp}
	0<f(t,u)\leq \partial_uf(t,0)u \,\hbox{ for all } \, (t,u)\in\R\times (0,1).
\end{equation}
In this case, for each $T>0$, problem \eqref{eq-main} admits a spreading speed, denoted by $c_T^*$, which characterizes the spreads of solutions of the Cauchy problem with compactly supported initial data (see \cite{hs,shen5}). This speed is linearly determined, and has the following variational formula
\begin{equation*}%\label{spreding}
c_T^*= \inf_{\mu>0} \frac{ \displaystyle \lim_{t\to+\infty} \frac{1}{t}\int_0^t \left(\partial_uf(\tau/T,0)+\mu^2\right)d\tau }{\mu}=\inf_{\mu>0} \frac{ \displaystyle \lim_{t\to+\infty} \frac{1}{t}\int_0^t \partial_uf(\tau,0)d\tau+\mu^2 }{\mu}.
\end{equation*}
Furthermore, it was proved in \cite{shen6} that there exists an almost periodic traveling wave solution with average speed $c$ if $c> c_T^*$, and such a solution does not exist if $c<c_T^*$ (the problem whether there exists a wave with the critical speed $c_T^*$ remains unclear). Clearly, by the above variational formula, $c_T^*$ is independent of $T$. This means that, for Fisher-KPP equation, temporal variations have no influences on the spreading speed.

Similar independence result can also be observed in some special bistable equations, although the unique average speed of bistable traveling wave is not linearly determined. Consider, for example, the following equation with a special cubic nonlinearity
\begin{equation}\label{bis-1-eq}
u_t=u_{xx}+u(1-u)(u-b(t/T)),\ \ t\in\R,\ x\in\R,  	
\end{equation}
where $b\in C^1(\R)$ is an almost periodic function with $0<\inf_{t\in\R} b(t) \leq \sup_{t\in\R}b(t)<1$.
It is straightforward to check that for each $T>0$,
$$u_T(t,x)=\psi_0\left( x- \int_0^t \sqrt{2}\left(\frac{1}{2}-b\left(\frac{s}{T}\right)\right)ds \right) $$
is an almost periodic traveling wave solution of \eqref{bis-1-eq} with wave speed $c'_T(t)=\sqrt{2}(1/2-b(t/T))$, where $\psi_0$ is a one variable function{\footnote{The function $\psi_0(\cdot)$ is actually the unique (up to shifts) wave profile of traveling wave solution for the homogeneous equation $u_t=u_{xx}+u(1-u)(u-b_0)$, where $b_0\in (0,1)$ is a constant. The associated wave speed is $\sqrt{2}(1/2-b_0)$.}} given by $\psi_0(x)=1/(1+\me^{x/\sqrt{2}})$ for $x\in\R$. 
As a consequence, the corresponding average speed is given by
$$\bar{c}_T= \lim_{t\to+\infty} \frac{1}{t} \int_0^t\sqrt{2}\left(\frac{1}{2}-b(s)\right)ds,$$
and clearly, it is independent of $T>0$.

%\begin{equation}\label{profile-psi0}
%\psi_0(x)=\frac{1}{1+\me^{x/\sqrt{2}}} \ \ \hbox{ for }\,x\in\R.
%\end{equation}

In light of the above observations, it is natural to ask whether $\bar{c}_T$ is independent of $T$ in general bistable or multistable equations. We will show that this is not true in general, by determining the asymptotic limits of $\bar{c}_T$ as $T\to 0^+$ and as $T\to+\infty$, respectively. As matter of matter, one can see from our main results  (see Theorems \ref{th-rapid} and \ref{th-slow} below) that the forms of the two limiting speeds are different.
Furthermore, at the end of this section, we will provide a time periodic bistable example demonstrating that the limiting speeds exhibit opposite signs, which means that temporal variations can even alter the direction of wave propagation.

Regarding the problem of the existence of limits for average speeds, to our knowledge, only some partial results in the periodic case are available in the literature. More precisely, the homogenization limit (as $T\to 0^+$) of bistable wave speed was determined in \cite{contri}, while the question whether the speed converges as $T\to+\infty$ remains open. On the other hand, for temporally homogeneous but spatially periodic equation 
\begin{equation}\label{eq-space}
		w_t=w_{xx}+f\left(x/T,w\right),\ \ t\in\R,\ x\in\R,
\end{equation}
where $f(x,\cdot)$ is periodic in $x$, the problem has been better understood. As a matter of fact, the limits of wave speeds for pulsating waves (which is an extension of traveling waves in spatially periodic media, analogously to time periodic traveling waves defined above) as $T\to 0^+$ or as $T\to+\infty$ were established for equations with various types of nonlinearties, and some interesting consequences regarding the effects of spatial heterogeneity were obtained, see e.g., \cite{ehr,sk,hfr,hnr,he,hps,dhl,dhz2,dhz1,dll}.
Note that the aforementioned works are all concerned with periodic equations. Our main results for the limits of average wave speeds appear to be the first one under almost periodic setting. As we will elaborate below, new approaches are needed to address the almost periodic problems.

\subsection{Asymptotic limit of average speed in rapidly oscillating environment}
Let us first establish the homogenization limit of the average speed $\bar{c}_T$ under the assumption of the existence of traveling wave for the homogenized equation. More precisely,
denoting by $\bar{f}$ the average of the function $f(t,u)$ with respect to $t$, that is,
\begin{equation}\label{mean-f}
	\bar{f}(u)=\lim_{t\to+\infty} \frac{1}{t}\int_0^t f(\tau,u)d\tau\,\,\hbox{ for }\, u\in\R,
\end{equation}
we consider the following homogeneous equation
\begin{equation}\label{mean-eq}
	u_t=u_{xx}+\bar{f}(u),\ \ t\in\R,\ x\in\R.
\end{equation}
We assume that

\begin{itemize}
	\item[(A4)] Equation \eqref{mean-eq} admits a traveling wave solution $u(t,x)=\phi_0(x-c_0t)$ connecting $0$ and $1$, that is, $(\phi_0,c_0)$ satisfies
	\begin{equation}\label{tw-homo}\left\{\baa{l}
		\phi_0''(\xi)+c_0\phi_0'(\xi)+\bar{f}(\phi_0(\xi))=0 \quad\hbox{for }\,\,\xi\in\R, \vspace{5pt}\\
	    \phi_0(-\infty)=1,\quad \phi_0(+\infty)=0. \eaa\right.
	\end{equation}
\end{itemize}
Thanks to (A2), it is clear that $0$ and $1$ are two linearly stable zeros of $\bar{f}$. It is well known (see e.g., \cite{fm}) that the wave profile $\phi_0$ is unique (up to shifts) and decreasing, the wave speed $c_0$ is unique, and the wave $\phi_0(x-c_0t)$ is globally stable.

\begin{theo}\label{th-rapid}
Let {\rm (A1)-(A4)} hold, and let $\bar{c}_T$ be the average speed defined in \eqref{mean-speed} and $c_0$ be the unique wave speed of \eqref{tw-homo}. Then there holds 
\begin{equation}\label{con-cT-c0}
	\lim_{T\to 0^+} \bar{c}_T=c_0.
\end{equation}
Furthermore, if 
 \begin{equation}\label{assume-F} 
	F(t,u), \,\partial_u F(t,u),\, \partial_{uu} F(t,u) \hbox{ are almost periodic in $t$ uniformly for $u$ in bounded sets,}  
\end{equation}
 where the function $F:\R\times\R\to \R$ is defined by
\begin{equation}\label{define-F}
	F(t,u)=\int_0^t (f(\tau,u)-\bar{f}(u)) d\tau \ \ \hbox{ for }\, t\in\R,\, u\in\R,
\end{equation}
   then there exists a constant $M_1>0$ independent of $T$ such that
    \begin{equation}\label{es-con-rapid}
    	|\bar{c}_T-c_0 | \leq M_1T  \,\,\hbox{ for all small }\,\, T>0.
    \end{equation}
\end{theo}

The convergence of average speed in \eqref{con-cT-c0} extends the homogenization result in \cite{contri} for periodic problem to almost periodic one. However, it seems difficult to employ the method used in \cite{contri} to address our almost periodic problem. To explain that new ideas are needed, let us briefly recall the strategy of the proof in \cite{contri} (a similar idea also appeared in an earlier paper \cite{he} which is concerned with the homogenization of pulsating waves for spatially periodic equations of combustion type). As a matter of fact, in the periodic case, the wave profile $V_T(t,\xi)$ and the speed $d_T$ satisfy
\begin{equation*}
	\partial_t V_T= \partial_{\xi\xi} V_T+ d_T\partial_{\xi} v_T+f(t/T,V_T)\ \ \hbox{ for }\, t\in\R,\,\,\xi\in\R.
\end{equation*}
The work \cite{contri} proved the convergence of $d_T$ together with that of $V_T(t,\xi)$ as $T\to 0^+$ by applying parabolic estimates and compactness arguments to the above equation. The proof highly relied on two ingredients: the constant speed $d_T$ is bounded in $T>0$ and the rescaled wave profiles $V_T(Tt,\xi)$ share the same period with respect to $t$ for all $T>0$. Therefore, to apply this approach to our almost periodic equation, one needs to 
establish the uniform boundedness of the function $t \to c_T'(t)$ for all small $T>0$, and the uniform almost periodicity of the rescaled wave profile $U_T(tT,\xi)$ in $t\in\R$ with respect to all small $T$. Both are very difficult to reach due to the complexity of almost periodic problems.

We provide a new approach to prove the convergence of average speed $\bar{c}_T$, without knowing much information on the profile $U_T(t,\xi)$ when $T\to 0^+$, and it also allows the nonlinearity to be of multistable type. Our approach is based on the global stability of almost periodic traveling wave as stated in Theorem \ref{th-known} (ii) above, and the construction of super- and sub-solution of \eqref{eq-main} where the key point is to appropriately perturb the traveling wave of the homogenized equation \eqref{mean-eq}. Inspired by the averaging theory for the Cauchy problem of ODEs (see e.g., \cite{moser,perko}) and functional differential equations (see e.g., \cite{hv}) with high frequency forcing, we choose $TF(t,u)$ as a perturbation function, where $F$ is defined by \eqref{define-F}. It turns out that this function works effectively, provided that $F(t,u)$, $\partial_u F(t,u)$ and $\partial_{uu} F(t,u)$ are almost periodic (or equivalently, they are bounded) in $t$ uniformly with respect to $u$ in bounded sets, i.e., the additional assumption \eqref{assume-F} holds. This assumption in particular implies that the perturbation term is of order $T$ as $T\to 0^+$. As a consequence, we are able to obtain the convergence of $\bar{c}_T$, and the estimate of the convergence rate as stated in \eqref{es-con-rapid} as well. Note that \eqref{es-con-rapid} implies that the deviation of $\bar{c}_T$ in comparison to  $c_0$ is at most linear along with small $T>0$.

However, it is known that the assumption \eqref{assume-F}  does not hold for a large class of $f$ (see e.g., \cite{fink}). On the other hand, it holds if $f(t,\cdot)$ is quasi-periodic with the periods satisfying a nonresonance condition. This follows from a classical result by Moser \cite{moser}, and  we recall it below for clarity.

\begin{pro} {\rm (\cite{moser})}
Assume that the function $f(t,u)$ is quasi-periodic in $t$ with basic frequencies $\omega=(\omega_1,\cdots,\omega_N)\in\R^N$ for some $N\geq 1$, in the sense that
$$f(t,u)=\hat{f}(\omega_1t,\cdots,\omega_Nt,u),$$
where $\hat{f}: \R^{N+1}\to \R$, $(\tau_1,\cdots,\tau_N,u)\mapsto \hat{f}(\tau_1,\cdots,\tau_N,u)$ is a $C^2$ function that is $1$-periodic in each $\tau_i$ for $1\leq i\leq N$. Assume further that there exist $\alpha>0$ and $\beta>0$ such that
\begin{equation}\label{nonresonance}
	|a \cdot \omega | \geq \beta |a|^{-\alpha} \,\,\hbox{ for all } \,\, a\in \Z^N\setminus\{0\}.
\end{equation}
Then the functions $F(t,u)$, $\partial_u F(t,u)$ and $\partial_{uu} F(t,u)$ are quasi-periodic in $t$ with basic frequencies $\omega$, where $F(t,u)$ is given by \eqref{define-F}.  Consequently, they are all almost periodic in $t$ uniformly with respect to $u$ in bounded sets.
\end{pro}

Without the assumption \eqref{assume-F}, we are still able to prove the convergence \eqref{con-cT-c0} by making an almost periodic approximation of $F(t,u)$  (see the approximating function $F_T(t,u)$ defined in \eqref{de-FT} below). We can prove that the perturbation term $TF_T(t,u)$ converges to $0$ as $T\to 0^+$ uniformly with respect to $t\in\R$ and $u$ in bounded sets, but it is unclear whether  $TF_T(t,u)$ is of order $T$ as $T\to 0^+$. Therefore, the question whether the estimate \eqref{es-con-rapid} remains hold in the general case is open.

We now come back to condition (A4). In the following proposition, we collect some general conditions ensuring the existence of solutions to problem \eqref{tw-homo}.

\begin{pro}\label{multi-existence}
Assume that $\bar{f}(0)=\bar{f}(1)=0$ and define
\begin{equation*}%\label{define-G}
	G(u)=\int_{0}^{u}\bar{f}(s)ds \,\,\hbox{ for } \,\,u\in [0,1].
\end{equation*}
Then {\rm (A4)} holds true provided that one of the following conditions holds:
\begin{itemize}
	\item [{\rm (a)}] $G(1)<0$ and $\bar{f}(u)< 0$ for all $u\in \left\{u\in (0,1):\, G(u)> G(1)\right\}$;
	\item [{\rm (b)}] $G(1)>0$ and $\bar{f}(u)> 0$ for all $u\in \left\{u\in (0,1): \,G(u)> G(0)\right\}$;
	\item [{\rm (c)}] $G(1)=0$ and $G(u)<0$ for all $u\in (0,1)$.
\end{itemize}
\end{pro}
Note that condition (A2) is not used in the above proposition, and therefore the limiting states $0$ and $1$ can be degenerate zeros of $\bar{f}$. It is clear that the conditions (a), (b) and (c) also imply that the wave speed $c_0$ is
negative, positive, and zero, respectively. The above proposition follows easily from \cite[Proposition 1.1]{dll} and its proof. Indeed, each of the conditions (a)-(c) can guarantee that equation \eqref{mean-eq} admits a propagating terrace{\footnote{Roughly speaking, a propagating terrace of \eqref{mean-eq} is a finite family of stacked traveling waves with ordered wave speeds. We refer to \cite{dgm,po,dm,gr,dll-23} and references therein for the precise definition and more properties of propagating terraces for various types of reaction-diffusion equations.} connecting $0$ and $1$ (see \cite[Theorem 1.2]{po}), and such a terrace necessarily consists of a single traveling wave (see the proof of \cite[Proposition 1.1]{dll}).

\subsection{Asymptotic limit of average speed in slowly oscillating environment}
In this subsection, we present our result on the asymptotic limit of average speed in slowly oscillating environment.
As we mentioned earlier, even in the periodic case, whether the limit of $\bar{c}_T$ as $T\to +\infty$ exists is unknown. Let us first derive the limit from a heuristic point of view.
For any $T>0$ and $\alpha\in (0,1)$, denote by $X_T^{\alpha}(t)$ the $\alpha$-level set of the almost periodic traveling wave solution $u_T(t,x)$, that is,
$$u_T(t,X_T^{\alpha}(t))=\alpha\, \hbox{ for all }\, t\in\R.$$
When $T$ is large, the environment is closer and closer to homogeneous on large intervals, thus locally $u_L(t,x)$ may be close to a traveling wave of the homogeneous equation with the nonlinearity $f$ temporally frozen at some time $s\in\R$, that is,
\begin{equation}\label{frozen-eq-f}
	u_t=u_{xx}+f(s,u),\ \  t\in\R,\,\,x\in\R.
\end{equation}
This leads us to conjecture that the motion of $X_T^{\alpha}(t)$ can by approximately described by the following ODE
$$\frac{d X_T^{\alpha}}{dt}=c\left(t/T\right),\ \  \hbox{for }\, t\in\R,$$
where for each $s\in\R$, $c(s)$ denotes the wave speed of traveling wave (if exists) for the homogeneous equation \eqref{frozen-eq-f}. As it is known, the average speed of an almost periodic traveling wave should be equal to the mean speed of any level set, and thus for each $T>0$, $\bar{c}_T$ is equal to the mean speed of $X_T^{\alpha}(t)$.
This would lead to a prediction that, as $T\to +\infty$, $c_T$ converges to a constant $c_*$ which is given by
\begin{equation}\label{limit-c*}
	c_*=\lim_{t\to\infty} \frac{1}{t} \int_0^t c(s) ds,
\end{equation}
provided that the above limit exists.

From the above heuristic arguments, to establish the convergence of $\bar{c}_T$, it is necessary to make sure that \eqref{frozen-eq-f} admits a traveling wave connecting $0$ and $1$, and that the limit in \eqref{limit-c*} exists. The latter actually requires us to consider all equations in the hull of \eqref{frozen-eq-f} (see the reasons in Remark \ref{re-A5} below). In other words, for each $g\in H(f)$ where $H(f)$ denotes the hull of $f$ (we will recall the definition of the hull for almost periodic functions in Section 2.1), and for each $s\in\R$, we need to impose conditions on the following homogeneous equation
\begin{equation}\label{frozen-eq-g}
	u_t=u_{xx}+g(s,u),\ \   t\in\R,\,\,x\in\R.
\end{equation}

Our assumption in this part is stated as follows:

\begin{itemize}
	\item[(A5)] For each $g\in H(f)$, and each $s\in\R$, equation \eqref{frozen-eq-g} admits a traveling wave connecting $0$ and $1$.
\end{itemize}

Clearly, the wave profile and wave speed depend on the nonlinearity $g(s,\cdot)$.  In our discussion below, we use $\phi(\xi;g\cdot s)$ and $c(g\cdot s)$ to denote the profile and speed, respectively. Namely, $(\phi(\xi;g\cdot s), c(g\cdot s))$ satisfies
\begin{equation}\label{tw-frozen}\left\{\baa{l}
	\partial_{\xi\xi} \phi(\xi;g\cdot s)+c(g\cdot s)\partial_{\xi} \phi(\xi;g\cdot s)+g(s, \phi(\xi;g\cdot s))=0
   \quad\hbox{for }\,\,\xi\in\R, \vspace{5pt}\\
\phi(-\infty;g\cdot s)=1,\quad \phi(+\infty;g\cdot s)=0. \eaa\right.
\end{equation}
It is easily seen that under condition (A2), \eqref{asspars} remains hold with $f$ replaced by $g\in H(f)$. This implies that, for each  $g\in H(f)$ and $s\in\R$,  $u^+=1$ and $u^-=0$ are two linearly stable zeros of $g(s,\cdot)$, and consequently, the traveling wave $\phi(x-c(g\cdot s)t;g\cdot s)$ is unique (up to shifts) and globally stable (see \cite{fm}).
We will show in Section 4.1 that the functions $(g,s)\mapsto \phi(\cdot;g\cdot s)$ and $(g,s)\mapsto c(g\cdot s)$ are continuous and almost periodic in $s\in\R$. This in particular implies that the limit in \eqref{limit-c*} exists.

\begin{theo}\label{th-slow}
	Let {\rm (A1)-(A3)} and {\rm (A5)} hold. Then there holds
	\begin{equation}\label{es-con-slow}
		|\bar{c}_T-c_* | \leq \frac{M_2}{T}  \,\,\hbox{ for all large }\,\, T>0,
	\end{equation}
	for some constant $M_2$ independent of $T$,  where $c_*$ is given by \eqref{limit-c*}.
\end{theo}

The estimate \eqref{es-con-slow} implies that in slowly oscillating environments, the speed difference $\bar{c}_T-c_*$ varies at most linearly with respect to $1/T$,  analogously to \eqref{es-con-rapid} in rapidly oscillating case.

Theorem \ref{th-slow} will be proved by applying a similar strategy to that of Theorem \ref{th-rapid}, but here the super- and sub-solution will be constructed by using the family of homogeneous traveling waves $(\phi(\xi;f\cdot s), c(f\cdot s))_{s\in\R}$. The proof uses as an essential ingredient the property that the wave profile $\phi(\xi;f\cdot s)$ is of class $C^1$ in $s\in\R$ and $\partial_s\phi(\xi;f\cdot s)$ is bounded for all $s\in\R$ and $\xi\in\R$ (see Proposition \ref{smooth-s} below). We point out that, a similar idea was used in a recent paper by the author and Hamel, Liang \cite{dhl}, which determined the asymptotic limits of wave speed and wave profile for bistable pulsating wave in slowly periodically oscillating media (see \cite[Theorem 1.6]{dhl}). A stronger assumption on $f$ ($f$ is homogeneous near $u=0$ and $u=1$){\footnote{Without this additional assumption, the limiting speed of slow oscillations was also determined in \cite[Theorem 1.2]{dhl} by an alternative approach, which is based on the method of large-time and large-space scaling and some earlier results in \cite{g} concerning the asymptotic behavior of solutions for the rescaled equation.}} was used in \cite{dhl} to construct super- and sub-solution, while such an assumption is not needed in our theorem. Moreover, as we will discuss in Remark \ref{rm-space-period}, our convergence result for temporally heterogeneous equation exhibits some specific features when compared to that for spatially periodic equation.

\begin{rem}\label{re-A5}
 It should be pointed out that condition {\rm (A5)} is necessary in our proof, although the limiting speed $c_*$ does not involve with functions in $H(f)$. As a matter of fact, having in hand only the existence of $(\phi(\xi;f\cdot s), c(f\cdot s))$, one may be unable to show the almost periodicity (or even the continuity) of the function $s\mapsto c(f\cdot s)$, and derive some uniform estimates of $\phi(\xi;f\cdot s)$ with respect to $s$, which will play an important role in our construction of super- and sub-solution. We also mention that {\rm (A5)} holds under some general conditions on each $g(s,\cdot)$, analogously to those collected in {\rm Proposition \ref{multi-existence}}.
\end{rem}

\begin{rem}\label{rm-space-period}
	Consider the spatially periodic equation \eqref{eq-space} with $f$ being of bistable form in the sense that, in addition to {\rm (A2)}, there exists some $b\in C^1(\R,(0,1))$ such that for each $x\in\R$,  
	\begin{equation}\label{bis-form}
		f(x,0)=f(x,1)=f(x,b(x))=0,\ f(x,\cdot)<0\hbox{ on }(0,b(x)),\ f(x,\cdot)>0\hbox{ on }(b(x),1).
	\end{equation}
  The existence of pulsating waves for small or large $T>0$, as well as the asymptotic limits of wave speed (denoted by $e_T$), were established in {\rm \cite{dhl,dhz2,dhz1}}. In particular, it is known that as $T\to 0^+$, $e_T$ converges to the wave speed for the corresponding homogenized equation ({\rm \cite[Theorems 1.2 and 1.4]{dhz1}}), which is similar to our {\rm Theorem \ref{th-rapid}}. 
	
	On the other hand, the limiting speed of slow oscillations for equation \eqref{eq-space} exhibit a different form compared to {\rm Theorem \ref{th-slow}} for equation \eqref{eq-main}. Indeed, in the spatially periodic case, if $\int_{0}^1 f(x,u)du$ has a fixed nonzero sign for all $x\in\R$, then as $T\to+\infty$, $e_L$ converges to the harmonic mean (over one period) of the speed function $s\mapsto c(s)$ which is determined by the family of homogeneous equations \eqref{frozen-eq-f} ({\rm \cite[Theorem 1.2]{dhl}}), and if the function $x\mapsto \int_{0}^1 f(x,u)du$ changes signs, then wave blocking phenomena happens for all large $T$, i.e, $e_T=0$ ({\rm \cite[Theorem 1.7]{dhz2}}); while for our temporally heterogeneous equation \eqref{eq-main}, it follows from {\rm Theorem \ref{th-slow}} that the average speed $\bar{c}_T$ converges to the arithmetic mean of the function $s\mapsto c(s)$ no matter whether $t\mapsto \int_{0}^1 f(t,u)du$ changes signs or not.     
	
	Furthermore, {\rm Theorems \ref{th-rapid}} and {\rm \ref{th-slow}} indicate that the dependence of the sign of $\bar{c}_T$ with respect to $T$ is more intricate than that of $e_T$. Indeed, one will see from the example in the next subsection that, $\bar{c}_T$ has opposite signs for different values of $T$;  while for equation \eqref{eq-space},  if $e_T \neq 0$, then it has the sign of $\int_0^l\int_{0}^1 f(x,u)dudx$, where $l>0$ is the period of $f(x,u)$ with respect to $x$, and consequently, the sign of  $e_T$ is independent of $T$ ({\rm \cite[Theorem 1.1]{dhz1}}).     
\end{rem}

\begin{rem}\label{re-existence}
When condition {\rm (A4)} (resp. {\rm (A5)}) holds, it is natural to ask whether an almost periodic traveling wave solution of \eqref{eq-main} exists for all small $T>0$ (resp. large $T>0$). This is a difficult problem, since the dynamics of the ODE \eqref{ode-f} between the stable steady states $0$ and $1$ can be very complicated. In particular, unlike the periodic case, there may not exist any almost periodic solution lying strictly between $0$ and $1$ (see e.g., {\rm \cite{shen3,sy}}).
\end{rem}

\subsection{Effects of temporal variations in  periodic case}
After determining the limiting speeds for average speed of almost periodic traveling wave for small or large $T$, we present some interesting consequences regarding the effects of temporal variations. For simplicity, we focus on the periodic environment, and consider the bistable reaction-diffusion equation of form 
\begin{equation}\label{eq-p-bis}
		v_t=v_{xx}+a\left(t/T\right)v\left(v-b\left(t/T\right)\right)(1-v),\ \ t\in\R,\ x\in\R, 
\end{equation}  
where the functions $s\mapsto a(s)$ and $s\mapsto b(s)$ are of class $C^1(\R)$, 1-periodic, and $a$ is positive, $b$ ranges in $(0,1)$. For each $T>0$, it follows from \cite{abc} that equation \eqref{eq-p-bis} admits a unique (up to spatial shifts) periodic traveling wave
$$v_T(t,x)=V_T(t,x-d_Tt)\ \   \hbox{ for }\, t\in\R,\, x\in\R, $$
connecting $0$ and $1$,  where $V_T(t,x)$ is $T$-periodic in $t$, decreasing in $x$, and the speed $d_T\in\R$ is a constant.

The following corollary is an easy consequence of Theorems \ref{th-rapid} and \ref{th-slow}.

\begin{cor}\label{cor-bistable-limit}
For each $T>0$, let $d_T\in\R$ be the unique speed of time periodic traveling wave for \eqref{eq-p-bis}. There holds $\lim_{T\to 0^+} d_T=d_0$, and $\lim_{T\to+\infty} d_T=d_*$, where
\begin{equation}\label{speed-d0}
	d_0=\sqrt{2\int_0^1 a(s)ds}\left (\frac{1}{2}-\frac{\displaystyle\int_0^1a(s)b(s)ds}{\displaystyle\int_0^1 a(s)ds}\right),
\end{equation}
and
\begin{equation}\label{speed-d*}
	d_*=\int_0^1  \sqrt{2a(s)} \left(\frac{1}{2}- b(s)\right) ds.
\end{equation}
\end{cor}
%Here, we use the fact that for given $a_0>0$ and $b_0\in (0,1)$, $v(t,x)=\psi_0(\sqrt{a_0}(x-\sqrt{2a_0}(1/2-b_0)t)$ is a traveling wave of the homogeneous equation $v_t=v_{xx}+a_0v(1-v)(v-b_0)$, where $\psi_0(\cdot)$ is the function given in \eqref{profile-psi0}.

Furthermore, we have the following observations.

\begin{pro}\label{pro-effects}
Let $d_0$ and $d_*$ be the constants given by \eqref{speed-d0} and \eqref{speed-d*}, respectively. Then, the following statements hold true:
\begin{itemize}
	\item [{\rm (i)}] If $a=a_0$ is a positive constant, then $d_T=d_0=d_*$ for all $T>0$;
	\item [{\rm (ii)}] If $b=b_0\in (0,1)$ is a constant, then $d_0\geq d_*>0$ when $0<b_0<1/2$, and $d_0\leq d_*<0$ when $1/2<b_0<1$; moreover, the strict inequalities hold true as soon as $a(\cdot)$ is not a constant;
	\item [{\rm (iii)}] There exist $1$-periodic functions $a\in C^1(\R,(0,\infty))$ and $b\in C^1(\R,(0,1))$ such that $d_0$ and $d_*$ have different signs. 
\end{itemize} 	
\end{pro}

Statement (i) follows directly from the discussion on equation \eqref{bis-1-eq}, and statement (ii) is also easily verified.  One can observe from (ii) that when $b=b_0$ is a constant and $b_0\neq 1/2$, the propagation of equation \eqref{eq-p-bis} exhibits a faster wave speed when associated with high-frequency oscillations, in contrast to its propagation with low-frequency ones. This leads us to conjecture that, in this special case, $|d_T|$ is nonincreasing with respect to $T>0$, and consequently, $d_T$ has a fixed sign independent of $T>0$, which is the sign of $\int_0^1\int_0^1 f(t,u)dudt$.  

On the other hand, statement (iii) implies that,  $d_T$ cannot have a fixed sign for all $T>0$ in general. 
As a matter of fact, based on Corollary \ref{cor-bistable-limit}, it is not difficult to find examples such that $d_0$ and $d_*$ exhibit opposite signs. Below, we provide an example by some approximation arguments. More precisely, for each $T>0$ and $n\in\N$, we consider the following equation
\begin{equation}\label{eq-bistable-n}
	\partial_t v_n= \partial_{xx}v_n+\tilde{a}_n\left(t/T\right)v_n\left(v_n-\tilde{b}\left(t/T\right)\right)(1-v)\ \  \hbox{ for }\, t\in\R,\, x\in\R,
\end{equation}
where 
$$\tilde{b}(t)=\frac{1}{2}+\frac{1}{4}sin(2\pi t) \ \ \hbox{ for }\, t\in\R,$$
and $\tilde{a}_n\in C^1(\R)$ is positive, $1$-periodic, and as $n\to+\infty$, it converges pointwise to the following step function 
\begin{equation*}\tilde{a}(t)=\left\{\baa{ll}
	1  & \hbox{if }\,  t\in (k,1/4+k],\vspace{5pt}\\
	64/9 & \hbox{if }\,  t\in (k+1/4,k+1/2],\vspace{5pt}\\
	4  & \hbox{if }\,  t\in (k+1/2,k+1],\eaa  \quad \hbox{for }\, k\in\N. \right.
\end{equation*}
We may further require that the sequence $(\tilde{a}_n)_{n\in\N}$ is uniformly bounded. For each $T>0$ and $n\in\N$, denote by $d^n_T$ the unique wave speed of periodic traveling wave solution for \eqref{eq-bistable-n}. It then follows from  Corollary \ref{cor-bistable-limit} that for each $n\in\N$, $\lim_{T\to 0^+} d^n_T=d^n_0$ and $\lim_{T\to +\infty} d^n_T=d^n_*$, where $d^n_0$ and $d^n_*$ are, respectively, the limits defined as in \eqref{speed-d0} and \eqref{speed-d*} with $a(\cdot)$ and $b(\cdot)$ replaced by $\tilde{a}_n(\cdot)$ and $\tilde{b}(\cdot)$ respectively. Recall that $\tilde{a}_n(t)\to \tilde{a}(t)$ as $n\to\infty$ for $t\in\R$ and that $(\tilde{a}_n)_{n\in\N}$ is uniformly bounded. By using the Lebesgue dominated convergence theorem, one computes that
 $$\lim_{n\to\infty} d^n_0 = -\frac{1}{6\sqrt{290} \pi}<0 \quad\hbox{and}\quad \lim_{n\to\infty} d^n_*= \frac{\sqrt{2}}{24 \pi}>0.$$
Consequently, for sufficiently large $n$,  the signs of $d^n_0$ and $d^n_*$ are different. This example implies that,  
in sharp contrast with the result that the spreading speed $c^*_T$ is independent of $T$ under the Fisher-KPP assumption \eqref{fisher-kpp}, temporal variations significantly influence the wave propagation of bistable equations, as it does not only affect the quantities of speeds, but may also alter the direction of propagation.

%%%%%%%%%%%%%%%%%%%%%%%%%%%%%%%%%%%%%%%%
%%%%%%%%%%%%%%%%%%%%%%%%%%%%%%%%%%%%%%%%

\SE{Preliminary}

In this section, we present some preliminary materials for later use. We first recall the definition of almost periodic functions and collect some basic properties. Then we present a way of normalization for time almost periodic traveling waves. 

\subsection{Almost periodic functions}

We first recall Bohr's definition of almost periodic functions.

\begin{defi}\label{df-almost}
A continuous function $f:\R\to \R$, $t\mapsto f(t)$ is said to be almost periodic if, for any $\epsilon>0$, the following set $$\Omega(\epsilon)=\{\tau:\|f(\cdot+\tau)-f(\cdot)\|_{L^{\infty}(\R)}<\epsilon\}$$
is relatively dense in $\R$. Furthermore, a continuous function $f:\R\times\R\to \R$, $(t,u)\mapsto f(t,u)$ is said to be almost periodic in $t$ uniformly with respect to $u$ in bounded sets (resp. $u\in\R$) if $f$ is almost periodic in $t$ for each $u\in\R$, and is uniformly continuous on $\R\times E$ for any bounded subset $E\subset \R$ (resp. $E=\R$).	
\end{defi}

It is well known that almost periodic functions can be equivalently defined as follows.

\begin{lem}\label{eqdf-almost}
The following statements hold true:
\begin{itemize}
	\item [{\rm (i)}] A function  $f\in C(\R,\R)$ is almost periodic if and only if for any sequence $(t'_n)\subset \R$ and $(s'_n)\subset \R$, there are subsequences $(t_n)\subset (t'_n)$ and $(s_n)\subset (s'_n)$ such that
	$$\lim_{m\to\infty}\lim_{n\to\infty} f(t+t_n+s_m)=\lim_{n\to\infty} f(t+t_n+s_n)\,\hbox{ pointwise in }\,t\in\R;$$
	
	\item [{\rm (ii)}] A function  $f\in C(\R\times\R,\R)$ is almost periodic function in $t$ uniformly with respect to $u$ in bounded sets (resp. $u\in\R$) if and only if for any sequence $\{t_n'\}\subset \R$, there exists a subsequence $\{t_n\} \subset \{t_n'\}$ such that $\lim_{n\to\infty}f(t+t_n,u)$ exists uniformly for $t\in\R$ and $u$ in bounded sets (resp. $u\in\R$).
\end{itemize}	
\end{lem}
\begin{proof}
This lemma follows directly from \cite[Theorems 1.17, 2.7 and 2.8]{fink}.
\end{proof}

Let  $f\in C(\R\times\R,\R)$ be an almost periodic function in $t$ uniformly with respect to $u$ in bounded sets.
Let  $H(f)$ be the {\it hull} of the function $f$, which consists of functions $g\in C(\R\times\R,\R)$ such that
$$ f(t+\tau_n,u) \to g(t,u) \,\hbox{ as }\, n\to\infty\, \hbox{ uniformly in }\,t\in\R \,\hbox{ and }\, u\,\hbox{ in bounded sets} $$
for some $(\tau_n)\subset\R$. It is known that $H(f)$ is compact and metrizable under the compact open topology (see e.g., \cite{fink,shen1,sy}). Moreover, we have the following lemma.

\begin{lem}\label{almost-partial}
Let $f\in C(\R\times\R,\R)$ be an almost periodic function in $t$ uniformly with respect to $u$ in a subset $E\subset  \R$. Assume that $\partial_u f$, $\partial_{uu}f$ exist and they are bounded in $t\in\R$ and $u\in E$. Then for each $g\in H(f)$, $\partial_ug$ exists, and it is almost periodic in $t$ uniformly with respect to $u\in E$. Moreover, $g$, $\partial_ug$ are bounded in $t\in\R$ and $u\in E$ uniformly for $g\in H(f)$.
\end{lem}

\begin{proof}
Since $\partial_{uu} f(t,u)$ is bounded in $t\in\R$, $u\in E$, it follows that $\partial_{u} f(t,u)$ is uniformly continuous in $t\in\R$, $u\in E$ and that
$$ \lim_{h\to 0} \frac{f(t,u+h)-f(t,u)}{h}=\partial_u f(t,u) \,\hbox{ uniformly in }\, t\in\R,\,u\in E.$$
This together with the almost periodicity of $f$ implies that $\partial_{u} f(t,u)$ is almost periodic in $t$ uniformly with respect to $u\in E$.
Then, for each $g\in H(f)$, by Lemma \ref{eqdf-almost} (ii), there exists a sequence $(t_n)\subset \R$ such that
 $\lim_{n\to\infty} f(t+t_n,u)=g(t,u)$ and $\lim_{n\to\infty} \partial_u f(t+t_n,u)=g_1(t,u)$ uniformly in $t\in\R$ and $u\in E$
for some $g_1\in C(\R\times E)$. This immediately implies that $g_1$ is almost periodic in $t$ uniformly with respect to $u\in E$, and that $g_1=\partial_u g$. Furthermore, since $f$ and $\partial_u f$ are bounded in $t\in\R$ and $u\in E$, it is clear that $g$ and $\partial_u g$ are bounded in $t\in\R$ and $u\in E$ uniformly for $g\in H(f)$.
\end{proof}

\subsection{Normalization of almost periodic traveling waves}

For each    $T>0$, let $u_T(t,x)=U_T(t,x-c_T(t))$ be an almost periodic traveling wave solution of \eqref{eq-main}.
In Theorem \ref{th-known} (iii), we have stated that the wave profile $U_T(t,x)$ is unique up to spatial shifts by $C^1$ time almost periodic functions. The following lemma shows that we can normalize $U_T(t,x)$ for definiteness by requiring that
\begin{equation}\label{normal}
	U_T(t,0)=\frac{1}{2} \, \hbox{ for all } \, t\in\R.
\end{equation}
Clearly, with such a normalization, $U_T$ is uniquely determined.

\begin{lem}\label{lem-normal}
 Let {\rm (A1)-(A3)} hold. For each $T>0$, there exists a unique $C^1$ almost periodic function $t\mapsto z_T(t)$ such that the function $\tilde{U}_T(t,x):=U_T(t,x+z_T(t))$ is a wave profile of the almost periodic traveling wave solution $u_T(t,x)$ satisfying $\tilde{U}_T(t,0)=1/2$.
\end{lem}

\begin{proof}
Since $\partial_{x}U_L(t,x)<0$ in $\R^2$ by Theorem \ref{th-known} (i), and since $U_L$ is at least of class $C^1(\R)$, the implicit function theorem yields the existence of a uniquely determined $C^1(\R)$ function $t\in\R\mapsto z_T(t)$  such that
\begin{equation}\label{initial-normal}
U_L(t,z_T(t))=\frac{1}{2}\,  \hbox{  for all } \, t\in\R.
\end{equation}
By definition, we have $U_T(t,x)\to 1$ as $x\to-\infty$ and $U_T(t,x)\to 0$ as $x\to+\infty$ uniformly in $t\in\R$. This implies that the function $z_T(t)$ is bounded in $\R$, i.e., there exists some constant $C_T>0$ such that
$|z_T(t)|\leq C_T$ for all $t\in\R$. Furthermore, by the fact that $U_L(t,x)$ is decreasing in $x\in\R$ again, one finds some constant $\beta_T>0$ such that
\begin{equation}\label{bounded-deri}
	\partial_{x} U_L(t,x)  \leq -\beta_T  \,\hbox{ for all } \, t\in\R,\,-C_T\leq x\leq C_T.
\end{equation}

Next, we prove that the function $z_T(t)$ is almost periodic. For any sequence $\{t_k'\}_{k\in\N}\subset \R$, it follows from \eqref{initial-normal} and \eqref{bounded-deri} that
$$\baa{ll}
&|z_T(t+t'_m)-z_T(t+t'_n)| \vspace{5pt}\\
\leq &  \displaystyle  \frac{1}{\min_{s\in\R,\,x\in [-C_T,C_T]}|\partial_{x}U_T(s,x)|} \Big|\underbrace{U_T(t+t'_m,z_T(t+t'_m))}_{=1/2}-U_T(t+t'_m,z_T(t+t'_n))\Big| \vspace{5pt}\\
\leq & \displaystyle \frac{1}{\beta_T} \Big|\underbrace{U_T(t+t'_n,z_T(t+t'_n))}_{=1/2}-U_T(t+t'_m,z_T(t+t'_n))  \Big|   \eaa$$
for all $n,m\in\N$ and $t\in\R$.  Since the function $U_T(t,x)$ is almost periodic in $t$ uniformly in $x\in [-C_T,C_T]$, by Lemma \ref{eqdf-almost} (ii), there exists a subsequence $\{t_k\} \subset \{t_k'\}$ such that $\lim_{k\to\infty}U_T(t+t_k,x)$ exists uniformly for $t\in\R$ and $x\in [-C_T,C_T]$.
Thus, for any $\epsilon>0$, there exists $N(\epsilon)>0$ such that
$$ \left|U_T(t+t_n,z_T(t+t_n))-U_T(t+t_m,z_T(t+t_n))  \right| \leq \epsilon  $$
for all $n,m\geq N(\epsilon)$ and $t\in\R$, whence we have $|z_T(t+t'_m)-z_T(t+t'_n)| \leq \epsilon/\beta_T$. By Lemma \ref{eqdf-almost} (ii) again, this implies that $z_T(t)$ is almost periodic in $t\in\R$.

Finally, set $\tilde{U}_T(t,x)=U_T(t,x+z_T(t))$. Then we have $u_T(t,x)=\tilde{U}_T(t,x-\tilde{c}_T(t))$, where $\tilde{c}_T(t)=c_T(t)+z_T(t)$.  Clearly, $\tilde{U}_T(t,x)$ is almost periodic in $t$ uniformly with respect to
$u$ in bounded sets, and $\tilde{c}_T'(t)$ is almost periodic in $t$. In other words, $\tilde{U}_T(t,x)$ is a wave profile of the almost periodic traveling wave solution $u_T(t,x)$ with speed $\tilde{c}_T'(t)$. Moreover, it follows directly from \eqref{initial-normal} that $\tilde{U}_T(t,0)=1/2$. The proof of Lemma \ref{lem-normal} is thus complete.
\end{proof}

%%%%%%%%%%%%%%%%%%%%%%%%%%%%%%%%%%%%%%%%
%%%%%%%%%%%%%%%%%%%%%%%%%%%%%%%%%%%%%%%%

\SE{Convergence in rapidly oscillating environments}

This section is devoted to the proof of Theorem \ref{th-rapid} on the limit of average speed $\bar{c}_T$ as $T\to 0^+$. We first prove the convergence and estimate the convergence rate under the additional assumption \eqref{assume-F} in Section 3.1. For clarity, we restate the result of this part as follows. 
\begin{theo}\label{restate-th} 
	Let {\rm (A1)-(A4)} and {\rm \eqref{assume-F}} hold. Then the estimate \eqref{es-con-rapid} holds true.
\end{theo}
 
In Section 3.2, we generalize our arguments to prove the convergence of $\bar{c}_T$ to $c_0$ without the assumption  \eqref{assume-F}. Both parts rely on the construction of super- and sub-solution for the following Cauchy problem 
\begin{equation}\label{cauchy-rapid}
\left\{\baa{ll}
	\displaystyle \partial_t v_T= \partial_{xx}v_T+ f\left(t/T,v_T\right) & \hbox{for }\,  t>0,\, x\in\R, \vspace{5pt}\\
v_T(0,x)=\phi_0(x) & \hbox{for }\, x\in\R,
\eaa \right.
\end{equation}
where $\phi_0(\cdot)$ is the wave profile of the homogeneous wave \eqref{tw-homo}. Throughout this section, we denote by $v_T(t,x)$ the solution of problem \eqref{cauchy-rapid}, and assume that conditions (A1)-(A4) hold. Since $0<\phi_0<1$ in $\R$, it follows directly from the strong maximum principle of parabolic equations that $0<v_T(t,x)<1$ for all $t\geq 0$, $x\in\R$.

\subsection{Proof of Theorem \ref{restate-th}}

In this part, we assume that \eqref{assume-F} holds, i.e., $F(t,u)$, $\partial_u F(t,u)$ and $\partial_{uu} F(t,u)$ are almost periodic in $t$ uniformly with respect to $u$ in bounded sets. Before presenting the super- and sub-solution, 
let us first introduce some notations, which will be also useful in the proof of Theorem \ref{th-slow}. First of all, by condition (A2) and the boundedness of $\partial_{uu}f(t,u)$ in $(t,u)\in \R\times [-1,2]$, there exist two constants $\delta_1\in (0,\delta_0)$ and
$\gamma_1\in (0,\gamma_0)$ such that
\begin{equation}\label{choose-gamma1}
	\partial_uf(t,u) \leq -\gamma_1 \, \hbox{ for all } \, t\in\R,\,  u\in [-\delta_1,\delta_1]\cup [1-\delta_1,1+\delta_1].
\end{equation}
Since $\partial_uf(t,u)$ is bounded for $(t,u)\in \R\times [-1,2]$, there exists some constant $K_1>0$ such that
\begin{equation}\label{lip-cons}
	|f(t,u_1)-f(t,u_2)|\leq K_1|u_1-u_2|\, \hbox{ for all }\,  t\in\R,\,  u_1,u_2\in [-1,2].
\end{equation}

We first construct the super-solution of \eqref{cauchy-rapid} in the following lemma. 

\begin{lem}\label{sup-rapid}
For any $\epsilon_1\in (0,\delta_1/2)$, there exists $T_1=T_1(\epsilon_1)>0$ sufficiently small such that for any $T\in (0,T_1)$, then function $v^+_T:[0,+\infty)\times\R$ defined by
\begin{equation*}
v^+_T(t,x)=\phi_0\left(x-c_0t+\eta_T(t)\right)+q_T(t)+TF\left(\frac{t}{T}, \phi_0\left(x-c_0t+\eta_T(t)\right)+q_T(t)  \right)
\end{equation*}	
is a super-solution of \eqref{cauchy-rapid} for $t\ge 0$ and $x\in\R$, where $q_T(\cdot)$ and $\eta_T(\cdot)$ are $C^1([0,+\infty))$ functions satisfying
\begin{equation}\label{require-qt}
q_T(0)=\epsilon_1,\quad 0<q_T(t)\leq q_T(0),\quad -A_1\leq q'_T(t)<0  \quad\hbox{ for all }\,\, t\geq 0,	
\end{equation}
and
\begin{equation}\label{require-etat}
	\eta_T(0)=0,\quad  -A_2 \leq \eta_T'(t)<0 \quad\hbox{ for all }\,\, t\geq 0,	
\end{equation}
for some constants $A_1>0$ and $A_2>0$ independent of $T\in (0,T_1)$.
\end{lem}

\begin{proof}
By the definition of traveling waves of \eqref{mean-eq}, we have $0<\phi_0<1$ in $\R$, and for any $\epsilon_1\in (0,\delta_1/2)$, there exists $C_1>0$ such that
\begin{equation*}
	\left\{\baa{ll}
	\displaystyle 0< \phi_0(\xi)\leq \epsilon_1 & \hbox{ if } \,\xi\geq C_1,\vspace{5pt}\\
	\displaystyle 1-\epsilon_1 \leq \phi_0(\xi)<1 & \hbox{ if } \,\xi\leq -C_1.
	\eaa \right.
\end{equation*}
Since $\phi_0(\xi)$ is decreasing in $\xi\in\R$, there exists $\beta>0$ such that
\begin{equation}\label{de-phi}
	\phi_0(\xi) \leq -\beta \, \hbox{ for all }\,  -C_1\leq \xi\leq  C_1.
\end{equation}
To prove $v^+_T(t,x)$ is a super-solution, it suffices to find suitable $C^1([0,+\infty))$ functions $q_T$ and $\eta_T$  satisfying \eqref{require-qt} and \eqref{require-etat} respectively, such that if $T>0$ is sufficiently small, there holds
\begin{equation*}
	N(t,x):= \partial_t v^+_T(t,x)-\partial_{xx} v^+_T(t,x) -f\left( \frac{t}{T}, v^+_T(t,x) \right) \geq 0
\end{equation*}
for all $(t,x)\in [0,+\infty)\times\R$. Notice that $(\phi_0(\cdot),c_0)$ satisfies \eqref{tw-homo}, and that
$$\partial_t F\left(\frac{t}{T},u\right)=\frac{1}{T} \left(f\left(\frac{t}{T},u\right)-\bar{f}(u)\right) \,\hbox{ for all }\, t\in\R,\, u\in\R.$$
It is straightforward to check that for any $T>0$ and any $C^1([0,+\infty))$ functions $q_T$ and $\eta_T$, one has
$$\baa{ll}
N(t,x)  =\!\!\! & \displaystyle
\phi_0'(\xi)\eta_T'(t)+q_T'(t)+\bar{f}(\phi_0(\xi))-\bar{f}(\phi_0(\xi)+q_T(t)) \vspace{6pt}\\
&  \displaystyle \quad  +\, f\left(\frac{t}{T},\phi_0(\xi)+q_T(t) \right) - f\left(\frac{t}{T},\phi_0(\xi)+q_T(t) +TF\left(\frac{t}{T}, \phi_0(\xi)+q_T(t) \right)   \right)  \vspace{6pt}\\
 & \displaystyle \quad  +\, T \partial_uF\left(\frac{t}{T}, \phi_0(\xi)+q_T(t)\right) \left( \phi_0'(\xi)\eta_T'(t)+q_T'(t)-c_0\phi'(\xi) - \phi_0''(\xi)  \right) \vspace{6pt}\\
 & \displaystyle \quad  - \, T \partial_{uu}F\left(\frac{t}{T}, \phi_0(\xi)+q_T(t)\right) (\phi_0'(\xi))^2\eaa$$
for all $(t,x)\in [0,\infty)\times\R$, where $\xi=x-c_0t+\eta_T(t)$. 
Due to \eqref{lip-cons}, there holds
$$\baa{ll}
& \displaystyle \!\!\! \left| f\left(\frac{t}{T},\phi_0(\xi)+q_T(t) \right) -  f\left(\frac{t}{T},\phi_0(\xi)+q_T(t) +TF\left(\frac{t}{T}, \phi_0(\xi)+q_T(t) \right)\right) \right| \vspace{6pt}\\
\leq &  \displaystyle  \!\!\!   K_1 T \left|F\left(\frac{t}{T}, \phi_0(\xi)+q_T(t) \right) \right|  \eaa $$
for all $(t,x)\in [0,\infty)\times\R$. Clearly, $K_1>0$ is a constant independent of $T$.
Furthermore, since $F(s,u)$, $\partial_u F(s,u)$ and $\partial_{uu} F(s,u)$ are almost periodic in $s\in\R$ uniformly for $u\in [-1,2]$ by \eqref{assume-F}, they are bounded for $(s,u)\in \R\times [-1,2]$ (see e.g., \cite{fink}).
With $q_T$ required to satisfy \eqref{require-qt}, one has $0<\phi_0(\xi)+q_T(t)<1+\epsilon_1<1$ for all $(t,x)\in [0,\infty)\times\R$, whence there exists some $K_2>0$ (independent of $T>0$) such that
$$ \left|F\left(\frac{t}{T}, \phi_0(\xi)+q_T(t) \right) \right|,\,\, \left|\partial_uF\left(\frac{t}{T}, \phi_0(\xi)+q_T(t) \right) \right|,\,\,  \left|\partial_{uu}F\left(\frac{t}{T}, \phi_0(\xi)+q_T(t) \right) \right| \,\leq  K_2. $$
Finally, with $q_T$ and $\eta_T$  required to satisfy \eqref{require-qt} and \eqref{require-etat}, and with the fact that $\phi_0'(\cdot)$ and $\phi_0''(\cdot)$ are bounded in $\R$, one finds some $K_3>0$ (independent of $T>0$) such that
\begin{equation*}
\left|\phi_0'(\xi)\eta_T'(t)\right|+\left|q_T'(t)\right|+\left|c_0\phi'(\xi)\right| +\left|\phi_0''(\xi)  \right|+(\phi_0'(\xi))^2 \leq K_3
\end{equation*}
for all $(t,x)\in [0,\infty)\times\R$.  Combining the above estimates, one obtains that
\begin{equation}\label{sup-rapid-ntx}
N(t,x) \geq \phi_0'(\xi)\eta_T'(t)+q_T'(t)+\bar{f}(\phi_0(\xi))-\bar{f}(\phi_0(\xi)+q_T(t)) - TK_0
\end{equation}
for all $(t,x)\in [0,\infty)\times\R$, where $K_0=K_1K_2+K_2K_3$ is a constant independent of $T>0$.

We complete the proof by considering three cases: (a) $x-c_0t+\eta_T(t)\geq C_1$; (b) $x-c_0t+\eta_T(t)\leq -C_1$; (c) $-C_1<x-c_0t+\eta_T(t)<C_1$. In case (a), since $q_T$ is required to satisfy $0<q_T(t)\leq \epsilon_1$, we have $0<\phi_0(\xi) \leq \epsilon_1 <\delta_1/2$ and $0<\phi_0(\xi)+q_T(t) \leq 2\epsilon_1<\delta_1$, and hence, by \eqref{choose-gamma1}, there holds
$$\bar{f}(\phi_0(\xi))-\bar{f}(\phi_0(\xi)+q_T(t))   \geq \gamma_1q_T(t).$$
Since  $\eta_T(t)$ is required to be decreasing in $t$ and since $\phi_0(\xi)$ is decreasing in $\xi$, it then follows from \eqref{sup-rapid-ntx} that
$$N(t,x) \geq q_T'(t)+\bar{f}(\phi_0(\xi))-\bar{f}(\phi_0(\xi)+q_T(t)) - TK_0 \geq   q_T'(t)+\gamma_1q_T(t)-K_0T.$$
Let us now choose
\begin{equation}\label{choose-T1}
	T_1=\frac{\epsilon_1 \gamma_1}{K_0}>0,
\end{equation}
and the function $q_T(t)$ such that
\begin{equation*}
	q_T'(t)+\gamma_1q_T(t)-TK_0=0 \,\,\hbox{ for }\,\, t> 0,\quad\hbox{and}\quad  q_T(0)=\epsilon_1,
\end{equation*}
namely,
\begin{equation}\label{function-qt}
q_T(t)=\frac{K_0T}{\gamma_1}+\left(\epsilon_1-\frac{K_0T}{\gamma_1}\right) \me^{-\gamma_1t}\, \hbox{ for } \, t\geq 0.
\end{equation}
Clearly, for any $0<T< T_1$, the function $q_T(t)$ satisfies \eqref{require-qt} with $A_1=\epsilon_1\gamma_1$, and $N(t,x)\geq 0$ for all $(t,x)\in [0,\infty)\times\R$ such that  $x-c_0t+\eta_T(t)\geq C_1$, as soon as $\eta_T$ satisfies \eqref{require-etat}.

In case (b), with $q_T$ given by \eqref{function-qt}, one has $1-\delta_1< 1-\epsilon_1\leq \phi_0(\xi) < 1$ and $ 1-\delta_1 <\phi_0(\xi)+q_T(t) <1+\epsilon_1<1+\delta_1$.  Then, proceeding similarly as above, one can conclude that for any $0<T< T_1$, $N(t,x)\geq 0$ for all $(t,x)\in [0,\infty)\times\R$ such that  $x-c_0t+\eta_T(t)\leq -C_1$, provided that $\eta_T$ satisfies \eqref{require-etat}.

It remains to find a suitable function $\eta_T$ such that $N(t,x)\geq 0$ in case (c), i.e., $-C_1<x-c_0t+\eta_T(t)<C_1$. By the definition of $\bar{f}$ (see \eqref{mean-f}), one has
$$\left|\bar{f}(\phi_0(\xi))-\bar{f}(\phi_0(\xi)+q_T(t))\right|\leq K_1q_T(t), $$
where $K_1$ is the constant provided by \eqref{lip-cons}. Then, with $\eta_T(t)$ required to satisfy \eqref{require-etat}, it follows from \eqref{de-phi} and \eqref{sup-rapid-ntx} that
$$N(t,x) \geq -\beta\eta_T'(t)+q_T'(t)-K_1q_T(t) - TK_0.$$
Then, with $q_T(t)$ given by \eqref{function-qt}, by choosing $\eta_T(t)$ such that
$$-\beta\eta_T'(t)+q_T'(t)-K_1q_T(t) - TK_0=0,\quad\hbox{and}\quad \eta_T(0)=0,$$	
one has $N(t,x)\geq 0$ for all  $(t,x)\in [0,\infty)\times\R$ such that $-C_1<x-c_0t+\eta_T(t)<C_1$. It is easy to compute that
\begin{equation}\label{function-eta}
	\eta_T(t)=-\frac{\gamma_1+K_1}{\beta\gamma_1}\left(\epsilon_1-\frac{K_0T}{\gamma_1}\right)\left(1-\me^{-\gamma_1t}\right)-\frac{(\gamma_1+K_1)K_0T}{\beta\gamma_1}t \, \hbox{ for } \, t\geq 0.
\end{equation}
Furthermore, it is straightforward to check that
if $0<T<T_1$, then there holds
$$ -\frac{(\gamma_1+K_1)\epsilon_1}{\beta} \leq \eta'_T(t) \leq -\frac{(\gamma_1+K_1)K_0T}{\beta\gamma_1} <0 \, \hbox{ for all }\,  t\geq 0.$$
Therefore, the function $\eta_T(t)$ given by \eqref{function-eta} satisfies \eqref{require-etat} with $A_2=(\gamma_1+K_1)\epsilon_1/\beta$.

Combining the above, one can conclude that for any $0<T<T_1$, there holds $N(t,x)\geq 0 $ for all $(t,x)\in [0,+\infty)\times\R$. This completes the proof of Lemma \ref{sup-rapid}.
\end{proof}

The following lemma gives the sub-solution of \eqref{cauchy-rapid}.

\begin{lem}\label{sub-rapid}
For any $\epsilon_1\in (0,\delta_1/2)$, let $T_1=T_1(\epsilon_1)>0$ be given by~\eqref{choose-T1} as in Lemma~$\ref{sup-rapid}$. Then, for every $0<T<T_1$, the function $v^-_{T}:[0,+\infty)\times\R\to\R$ defined by
	$$v^-_{T}(t,x)=\phi_0\left(x-c_0t-\eta_T(t)\right)-q_T(t)+TF\left(\frac{t}{T}, \phi_0\left(x-c_0t-\eta_T(t)\right)-q_T(t) \right)$$
	is a sub-solution of~\eqref{cauchy-rapid} for $t\ge0$ and $x\in\R$, where $q_{T}$ and $\eta_{T}$ are $C^1([0,+\infty))$ functions given by~\eqref{function-qt} and~\eqref{function-eta}, respectively.
\end{lem}

\begin{proof}
	The proof is analogous to that of Lemma~\ref{sup-rapid}; therefore, we omit the details.
\end{proof}

We are now ready to give the

\begin{proof}[Proof of Theorem $\ref{restate-th}$]
For any $\epsilon_1\in (0,\delta_1/2)$, let $T_1=T_1(\epsilon_1)$ be given by~\eqref{choose-T1}.
Since the function $F(s,u)$ is bounded for $s\in\R$ and $u$ in bounded sets, there exists $T_2=T_2(\epsilon_1)\in (0,T_1]$ such that
\begin{equation}\label{find-T2}
|TF(s,u)|\leq \frac{\epsilon_1}{4}\, \hbox{ for all } \, s\in\R, \, u\in [-\epsilon_1,1+\epsilon_1],\, 0<T<T_2.
\end{equation}
Let $T\in (0,T_2)$ be arbitrary, and let $v_T^+(t,x)$ and $v_T^-(t,x)$ be the super-solution and sub-solution of \eqref{cauchy-rapid} given in Lemmas \ref{sup-rapid}-\ref{sub-rapid}, respectively. Since $q_T(0)=\epsilon_1$, $\eta_T(0)=0$, it is easily seen from \eqref{find-T2} that
$$ v_T^-(0,x)\leq  \phi_0(x)\leq v_T^+(0,x)\, \hbox{ for all } \, x\in\R. $$
Then, by the comparison principle for parabolic equations, we have
\begin{equation}\label{sup-sub-vt}
v_T^-(t,x)\leq  v_T(t,x)\leq v_T^+(t,x) \,\hbox{ for all } \, t\ge 0,\, x\in\R,
\end{equation}
where $v_T(t,x)$ is the solution of the Cauchy problem \eqref{cauchy-rapid}.

On the other hand, recall that $u_T(t,x)=U_T(t,x-c_T(t))$ is an almost periodic traveling wave solution of \eqref{eq-main}. By definition, $\lim_{x\to-\infty}U_T(0,x)=1$ and $\lim_{x\to+\infty}U_T(0,x)=0$, and thanks to Theorem \ref{th-known} (i), $U_T(0,x)$ is decreasing in $x\in\R$. Since $\lim_{x\to-\infty}\phi_0(x)=1$,  $\lim_{x\to+\infty}\phi_0(x)=0$, and $\phi_0(x)$ is decreasing in $x\in\R$, one finds some constant $h_T>0$ such that
$$U_T(0,x+h_T-c_T(0))-\delta' \leq v_T(0,x) \leq U_T(0,x-h_T-c_T(0))+\delta'\,\hbox{ for all }\, x\in\R,$$
where $\delta'\in (0,\delta_0/2)$ is the constant provided by Theorem \ref{th-known} (ii). Then, by the stability of almost periodic traveling wave solution of \eqref{eq-main}, there holds
\begin{equation}\label{stable-vt}
	\sup_{x\in\R} \left|v_T(t,x)-U_T(t,x-c_T(t)-\xi^*_T)\right|\to 0 \, \hbox{ as }\, t\to +\infty,
\end{equation}
for some constant $\xi^*_T\in\R$.

Combining \eqref{sup-sub-vt} and \eqref{stable-vt}, one can find some $t_0=t_0(T)>0$ sufficiently large such that
$$v_T^-(t,x)-\frac{\epsilon_1}{4} \leq v_T(t,x)-\frac{\epsilon_1}{4} \leq U_T(t,x-c_T(t)-\xi^*_T) \leq  v_T(t,x)+\frac{\epsilon_1}{4}\leq v_T^+(t,x)+\frac{\epsilon_1}{4}   $$
for all $t\geq t_0$ and $x\in\R$. Choosing $x:=x(t)=c_T(t)+\xi_T^*$, with the normalization condition \eqref{normal}, one obtains that
$$v_T^-(t,c_T(t)+\xi_T^*)-\frac{\epsilon_1}{4}\leq \frac{1}{2}\leq v_T^+(t,c_T(t)+\xi_T^*)+\frac{\epsilon_1}{4} \, \hbox{ for all }\, t\geq t_0. $$
Recall that $0<q_T(t)\leq q_T(0)=\epsilon_1$ for $t\geq 0$ by \eqref{require-qt}. It then follows from the definitions of $v_T^{\pm}$ and \eqref{find-T2} that
\begin{equation}\label{esti-phi0-vt}
\phi_0\left(c_T(t)+\xi_T^*-c_0t-\eta_T(t)\right)-\frac{3}{2}\epsilon_1 \leq \frac{1}{2} \leq \phi_0\left(c_T(t)+\xi_T^*-c_0t+\eta_T(t)\right)+\frac{3}{2}\epsilon_1
\end{equation}
for all $t\geq t_0$. Since $0<\epsilon_1<\delta_1/2<1/4$ and since $\lim_{x\to+\infty}\phi_0(x)=0$, $\lim_{x\to-\infty}\phi_0(x)=1$, \eqref{esti-phi0-vt} in particular implies that there exists a constant $\xi_0>0$ independent of $T$ such that
$$\left|c_T(t)+\xi_T^*-c_0t\right| \leq |\eta_T(t)|+\xi_0 \, \hbox{ for all } \, t\geq t_0.  $$
Remember that the function $\eta_T(t)$ is given by \eqref{function-eta}.
Multiplying both sides of the above inequality by $1/t$,  one obtains that
$$\baa{ll}
& \displaystyle \!\!\! \left|\frac{1}{t}\int_0^t c'_T(\tau)d\tau-c_0\right| \vspace{6pt}\\
\leq &  \displaystyle  \!\!\!  \frac{|\xi_T^*|+|c_T(0)|+\xi_0}{t}+ \frac{1}{t}\frac{\gamma_1+K_1}{\beta\gamma_1}\left(\epsilon_1-\frac{K_0T}{\gamma_1}\right)\left(1-\me^{-\gamma_1t}\right)
+\frac{(\gamma_1+K_1)K_0}{\beta\gamma_1}T  \eaa $$
for all $t\geq t_0$.
Passing to the limits as $t\to+\infty$ in both sides yields that
$$\left|\bar{c}_T-c_0\right| \leq \frac{(\gamma_1+K_1)K_0}{\beta\gamma_1}T.$$
Finally, we choose $M_1=(\gamma_1+K_1)K_0/(\beta\gamma_1)$. Clearly, this constant is independent of $T$. Since  $0<T<T_2$ is arbitrary, the proof of Theorem \ref{restate-th} is thus complete.
\end{proof}

%%%%%%%%%%%%%%%%%%%%%%%%%%%%%%%%%%%%%%%%

\subsection{Completion of the proof of Theorem \ref{th-rapid}}

In the case where the assumption \eqref{assume-F} does not hold, we use the following function to approximate $F(t,u)$: 
\begin{equation}\label{de-FT}
	F_T(t,u)=\int_{-\infty}^t \me^{-T(t-\tau)}\left(f(\tau,u)-\bar{f}(u)\right)d\tau \, \hbox{ for }\, t\in\R,\, u\in\R.
\end{equation}
Since $f$ is of class $C^2(\R^2)$, it is clear that $F_T$ is of class $C^2(\R^2)$. It is also easily seen that $F_T(t,\cdot)$ satisfies the following differential equation
\begin{equation}\label{eq-FT}
\frac{\partial F_T}{\partial t} (t,\cdot)=-TF_T(t,\cdot)+f(t,\cdot)-\bar{f}(\cdot)\, \hbox{ for }\,  t\in\R.
\end{equation}
Furthermore, we have the following lemma.

\begin{lem}\label{lem-FT}
Let $E$ be a bounded subset of $\R$. The function $F_T(t,u)$ defined by \eqref{de-FT} possesses the following properties: 
\begin{itemize}
\item [{\rm (a)}] There exists a continuous function $\chi:(0,+\infty)\to (0,+\infty)$ with $\lim_{T\to 0^+}\chi(T)=0$ such that $|F_T(t,u)|\leq \chi(T)T^{-1}$ for all $t\in\R$, $u\in E$, $T\in (0,\infty)$;
\item [{\rm (b)}]  $F_T(t,u)$, $\partial_uF_T(t,u)$ and $\partial_{uu}F_T(t,u)$ are almost periodic in $t\in\R$ uniformly with respect to $u\in E$ and $T$ in any compact subset of $(0,\infty)$;
 %and $\mathcal{M}(F_T(\cdot,u))$, $\mathcal{M}(\partial_u F_T(\cdot,u))$, $\mathcal{M}(\partial_{uu} F_T(\cdot,u)) \subset \mathcal{M}(f(\cdot,u))$;
\item [{\rm (c)}] $TF_T(t,u)$,  $T\partial_u F_T(t,u)$ and $T\partial_{uu} F_T(t,u)$ approach zero as $T\to 0^+$ uniformly with respect to $t\in\R$ and $u\in E$.
\end{itemize}
\end{lem}
\begin{proof}
This lemma follows directly from the proof of \cite[Lemmas 14.1-14.2]{fink}.
\end{proof}

Based on the above lemma, we are able to construct a pair of super-solution and sub-solution of \eqref{cauchy-rapid}, which are similar to Lemmas \ref{sup-rapid} and \ref{sub-rapid}.

\begin{lem}\label{sup-rapid-ge}
	For any $\epsilon_1\in (0,\delta_1/2)$, there exist $\tilde{T}_1=\tilde{T}_1(\epsilon_1)>0$ sufficiently small such that for any $T\in (0,\tilde{T}_1)$, the function $\tilde{v}^{\pm}_T:[0,+\infty)\times\R$ defined by
	\begin{equation*}
		\tilde{v}^+_T(t,x)=\phi_0\left(x-c_0t+\tilde{\eta}_T(t)\right)+\tilde{q}_T(t)+TF_T\left(\frac{t}{T}, \phi_0\left(x-c_0t+\tilde{\eta}_T(t)\right)+\tilde{q}_T(t)  \right),
	\end{equation*}	
     and
     	$$\tilde{v}^-_{T}(t,x)=\phi_0\left(x-c_0t-\tilde{\eta}_T(t)\right)-\tilde{q}_T(t)+TF_T\left(\frac{t}{T}, \phi_0\left(x-c_0t-\tilde{\eta}_T(t)\right)-\tilde{q}_T(t) \right),$$
	are, respectively, a super-solution and a sub-solution of \eqref{cauchy-rapid} for $t\ge 0$ and $x\in\R$, where $\tilde{q}_T(\cdot)$ and $\tilde{\eta}_T(\cdot)$ are $C^1([0,+\infty))$ functions satisfying
	\begin{equation}\label{choose-qt-1}
		\tilde{q}_T(0)=\epsilon_1,\quad 0<\tilde{q}_T(t)\leq \tilde{q}_T(0),\quad -\tilde{A}_1\leq q'_T(t)<0  \quad\hbox{ for all }\,\, t\geq 0,	
	\end{equation}
	and
	\begin{equation}\label{choose-etat-1}
		\tilde{\eta}_T(0)=0,\quad  -\tilde{A}_2 \leq \tilde{\eta}_T'(t)<0 \quad\hbox{ for all }\,\, t\geq 0,	
	\end{equation}
	for some constants $\tilde{A}_1>0$ and $\tilde{A}_2>0$ independent of $T\in (0,\tilde{T}_1)$.
\end{lem}
\begin{proof}
The proof is similar to that of Lemmas \ref{sup-rapid}-\ref{sub-rapid}, therefore we only provide the details when modifications are needed. We only show that $\tilde{v}^+_T(t,x)$ is a super-solution, since the sub-solution can be constructed in a similar way.

Let $\gamma_1>0$, $K_1>0$ and $\beta>0$ be the constants provided by \eqref{choose-gamma1}, \eqref{lip-cons} and \eqref{de-phi}, respectively.
Define
\begin{equation*}
	\tilde{N}(t,x):= \partial_t \tilde{v}^+_T(t,x)-\partial_{xx} \tilde{v}^+_T(t,x) -f\left( \frac{t}{T},\tilde{ v}^+_T(t,x) \right) \, \hbox{ for } \,(t,x)\in [0,+\infty)\times\R.
\end{equation*}
It suffices to prove that for small $T>0$, $\tilde{N}(t,x)\geq 0$ for all $(t,x)\in [0,+\infty)\times\R$. Thanks to \eqref{tw-homo} and \eqref{eq-FT}, one has
$$\baa{ll}
\tilde{N}(t,x)  =\!\!\! & \displaystyle
\phi_0'(\xi)\tilde{\eta}_T'(t)+\tilde{q}_T'(t)+\bar{f}(\phi_0(\xi))-\bar{f}(\phi_0(\xi)+\tilde{q}_T(t)) \vspace{6pt}\\
&  \displaystyle \quad  +\, f\left(\frac{t}{T},\phi_0(\xi)+\tilde{q}_T(t) \right) - f\left(\frac{t}{T},\phi_0(\xi)+\tilde{q}_T(t) +TF_T\left(\frac{t}{T}, \phi_0(\xi)+\tilde{q}_T(t) \right)   \right)  \vspace{6pt}\\
& \displaystyle \quad  +\, T \partial_uF_T\left(\frac{t}{T}, \phi_0(\xi)+\tilde{q}_T(t)\right) \left( \phi_0'(\xi)\tilde{\eta}_T'(t)+\tilde{q}_T'(t)-c_0\phi'(\xi) - \phi_0''(\xi)  \right) \vspace{6pt}\\
& \displaystyle \quad  - \, T \partial_{uu}F_T\left(\frac{t}{T}, \phi_0(\xi)+\tilde{q}_T(t)\right) (\phi_0'(\xi))^2
-TF_T\left(\frac{t}{T}, \phi_0(\xi)+\tilde{q}_T(t)\right) \eaa$$
for all $(t,x)\in [0,\infty)\times\R$, where $\xi=x-c_0t+\tilde{\eta}_T(t)$.
With $\tilde{q}_T$ and $\tilde{\eta}_T$ required to satisfy \eqref{choose-qt-1} and \eqref{choose-etat-1},
it follows from \eqref{lip-cons}, the boundedness of $\phi_0'(\cdot)$ and $\phi_0''(\cdot)$ in $\R$, and the boundedness of $F_T(s,u)$, $\partial_uF_T(s,u)$, $\partial_{uu}F_T(s,u)$ for $s\in\R$ and $u$ in bounded sets (due to Lemma \ref{lem-FT} (b)), that for each $T>0$, there holds
\begin{equation*}
	\tilde{N}(t,x) \geq \phi_0'(\xi)\tilde{\eta}_T'(t)+\tilde{q}_T'(t)+\bar{f}(\phi_0(\xi))-\bar{f}(\phi_0(\xi)+\tilde{q}_T(t)) - \tilde{K}_T
\end{equation*}
for all $(t,x)\in [0,\infty)\times\R$,
where $\tilde{K}_T$ is a positive constant given by
\begin{equation}\label{define-kT}
	\baa{ll}
	\tilde{K}_T =\!\!\! &  \sup_{s\in\R,\,\xi\in\R,\, u\in [-\epsilon_1,1+\epsilon_1]} T(K_1+1)  \left|F_T(s,u)\right| +T\left|\partial_{uu}F_T(s,u)  (\phi_0'(\xi))^2 \right| \vspace{6pt}\\
	&   \qquad\qquad\qquad\qquad\qquad +\,T\left|\partial_u F_T(s,u)\right| \left( (A_2+|c_0|) |\phi_0'(\xi)|+ A_1 + |\phi_0''(\xi)|\right).
	\eaa
\end{equation}
By Lemma \ref{lem-FT} (c), one has $\tilde{K}_T\to 0$ as $T\to 0^+$.
This in particular implies that there exists a small $\tilde{T}_1>0$ such that
\begin{equation*}%\label{choose-T1-1}
0<\tilde{K}_T<\epsilon_1\gamma_1 \, \hbox{ for all } \, 0<T<\tilde{T}_1.
\end{equation*}

Now, proceeding similarly as in the proof of Lemma \ref{sup-rapid}, for each $T\in (0,\tilde{T}_1)$, choosing
 $\tilde{q}_T$ and $\tilde{\eta}_T$ such that
\begin{equation*}
	\left\{\baa{ll}
	\displaystyle \tilde{q}_T'(t)+\gamma_1\tilde{q}_T(t)-\tilde{K}_T=0 \, \hbox{ for } \, t> 0 & \hbox{with } \,\, \tilde{q}_T(0)=\epsilon_1, \vspace{5pt}\\
	\displaystyle -\beta\tilde{\eta}_T'(t)+\tilde{q}_T'(t)-K_1\tilde{q}_T(t) - \tilde{K}_T=0  \, \hbox{ for } \, t> 0 & \hbox{with } \,\, \tilde{\eta}_T(0)=0,
	\eaa \right.
\end{equation*}
one can conclude that $\tilde{N}(t,x)\geq 0$ for all $(t,x)\in [0,\infty)\times\R$. It is straightforward to compute that
\begin{equation*}%\label{function-qt-1}
	\tilde{q}_T(t)=\frac{\tilde{K}_T}{\gamma_1}+\left(\epsilon_1-\frac{\tilde{K}_T}{\gamma_1}\right) \me^{-\gamma_1t}\, \hbox{ for } \, t\geq 0,
\end{equation*}
and
\begin{equation*}\label{function-eta-1}
	\tilde{\eta}_T(t)=-\frac{\gamma_1+K_1}{\beta\gamma_1}\left(\epsilon_1-\frac{\tilde{K}_T}{\gamma_1}\right)\left(1-\me^{-\gamma_1t}\right)-\frac{(\gamma_1+K_1)\tilde{K}_T}{\beta\gamma_1}t \, \hbox{ for } \, t\geq 0.
\end{equation*}
Clearly, the above functions satisfy \eqref{choose-qt-1} and \eqref{choose-etat-1} with $\tilde{A}_1=\epsilon_1\gamma_1$ and $\tilde{A}_2=((\gamma_1+K_1)\epsilon_1)/\beta$. This ends the proof of Lemma \ref{sup-rapid-ge}.
\end{proof}

We are now ready to complete
\begin{proof}[Proof of Theorem $\ref{th-rapid}$]
It remains to show \eqref{con-cT-c0} without the condition \eqref{assume-F}. Thanks to the preparations in Lemmas \ref{lem-FT}-\ref{sup-rapid-ge}, the proof is almost identical to that of Theorem \ref{restate-th}. Indeed, following the main lines as those used in the proof of Theorem \ref{restate-th}, with $TF(\cdot,\cdot)$, $q_T(\cdot)$ and $\eta_T(\cdot)$ replaced by $TF_T(\cdot,\cdot)$, $\tilde{q}_T(\cdot)$ and $\tilde{\eta}_T(\cdot)$ (which are provided by Lemma \ref{sup-rapid-ge}), respectively, one can conclude that
\begin{equation}\label{esti-ct-c0}
\left|c_T(t)+\xi_T^*-c_0t\right| \leq |\tilde{\eta}_T(t)|+\xi_0 \, \hbox{ for all large }\, t>0,
\end{equation}
where $\xi_T^*\in\R$ is a constant determined by the stability of $U_T$ (see \eqref{stable-vt}) and $\xi_0\in\R$ is a constant depending only on $\phi_0$ and $\epsilon_1$ (see \eqref{esti-phi0-vt}).
Multiplying both sides of \eqref{esti-ct-c0} by $1/t$, and then passing to the limits as $t\to+\infty$,
implies that
$$\left|\bar{c}_T-c_0\right| \leq \frac{(\gamma_1+K_1)}{\beta\gamma_1}\tilde{K}_T \, \hbox{ for all small }\, T>0,  $$
where $\tilde{K}_T$ is the constant defined in the proof of Lemma \ref{sup-rapid-ge} (see \eqref{define-kT}).
Then, in view of the fact that $\lim_{T\to 0^+}\tilde{K}_T=0$, taking the limits as $T\to 0^+$ in the above inequality, we obtain the convergence of $\bar{c}_T$ to $c_0$. This ends the proof.
\end{proof}	
%%%%%%%%%%%%%%%%%%%%%%%%%%%%%%%%%%%%%%%%
%%%%%%%%%%%%%%%%%%%%%%%%%%%%%%%%%%%%%%%%

\SE{Convergence in slowly oscillating environments}

This section is devoted to the proof of Theorem \ref{th-slow} on the convergence of $\bar{c}_T$ as $T\to+\infty$. 
Recall that for any $g\in H(f)$ and $s\in\R$, the couple $(\phi(\cdot;g\cdot s),c(g\cdot s))$ denotes the wave profile and wave speed for the homogeneous equation \eqref{tw-frozen}.  In Section 4.1, we collect some properties of $(\phi(\cdot;g\cdot s),c(g\cdot s))$ with respect to $g\in H(f)$ and $s\in\R$, which will be used in the construction of super- and sub-solution for equation \eqref{eq-main} when $T$ is large in Section 4.2. It should be pointed that, in the case where $f(s,u)$ is periodic in $s$ and of bistable form in the sense of \eqref{bis-form}, the main results stated in Section 4.1 have been proved in \cite[Proposition 4.1]{dhl}, but extra arguments are required to address the difficulties posed by the almost periodicity of $f$ and the possible existence of multiple zeros of $f(s,\cdot)$.

\subsection{Properties of traveling waves of frozen equations}
Throughout this subsection, let (A1), (A2) and (A5) hold, and we normalize each $\phi(\cdot;g\cdot s)$ uniquely in such a way that
\begin{equation}\label{normal-phi-g}
	\phi(0;g\cdot s)=\frac{1}{2}\, \hbox{ for each }\, g\in H(f),\, s\in\R.
\end{equation}
Notice that $\phi(+\infty;g\cdot s)=0$, $\phi(-\infty;g\cdot s)=1$, and that (A2) remains valid with  $f$ replaced by $g\in H(f)$. Namely, for each $g\in H(f)$ and $s\in\R$, there holds $g(s,0)=g(s,1)=0$, and
\begin{equation}\label{gsn-delta}
	\left\{\baa{ll}
	g(s,u)\le-\gamma_0 u & \hbox{for all } \,  u\in [-\delta_0,\delta_0],\vspace{5pt}\\
	g(t,u)\ge\gamma_0(1-u) & \hbox{for all }\,  u\in [1-\delta_0,1+\delta_0].  \eaa\right.
\end{equation}
Then, by classical theory on traveling wave for bistable or multistable equations, it is known that the wave profile $\phi(\cdot;g\cdot s)$ approaches the steady states $0$ and $1$ exponentially fast and the wave $\phi(x-c(g\cdot s)t;g\cdot s)$ is globally stable. As these properties will be frequently used later, we summarize them in the following lemma.

\begin{lem}\label{basic-phi}{\rm (\cite{fm})}
Let $(\phi(\cdot;g\cdot s),c(g\cdot s))$ be the unique solution of \eqref{tw-frozen} with the normalization condition \eqref{normal-phi-g}. Then the following statements hold true.
\begin{itemize}
	\item [{\rm (i)}] There exist some positive constants $A_1$ and $A_2$ (depending on $g$ and $s$) such that
  \begin{equation*}
  	\left\{\baa{ll}
  \phi(\cdot;g\cdot s) \sim A_1 \me^{-\lambda_1(g\cdot s)\xi}  & \hbox{as } \,\, \xi\to +\infty,\vspace{5pt}\\
  1- \phi(\cdot;g\cdot s) \sim A_2 \me^{\lambda_2(g\cdot s)\xi}  & \hbox{as } \,\, \xi\to -\infty,  \eaa\right.
  \end{equation*}
where $\lambda_1(g\cdot s)$ and $\lambda_2(g\cdot s)$ are positive constants given by
 \begin{equation}\label{decay-rate}
	\left\{\baa{l}
	\displaystyle   \lambda_1(g\cdot s)=\frac{c(g\cdot s)+\sqrt{(c(g\cdot s))^2-4\partial_ug(s,0)}}{2},\vspace{6pt}\\
    \displaystyle	\lambda_2(g\cdot s)=\frac{-c(g\cdot s)+\sqrt{(c(g\cdot s))^2-4\partial_ug(s,1)}}{2}.  \eaa\right.
\end{equation}
	\item [{\rm (ii)}] Let $u(t,x;g\cdot s,u_0)$ be the solution of the Cauchy problem of \eqref{frozen-eq-g} with initial function $u_0\in C(\R)\cap L^{\infty}(\R)$ satisfying
	$$\liminf_{x\to-\infty} u_0(x) \geq 1-\delta_0 \quad \hbox{and}\quad \liminf_{x\to +\infty} u_0(x) \leq \delta_0, $$
	where $\delta_0\in (0,1/2)$ is the small constant provided by {\rm (A2)}. There holds
	\begin{equation*}
		\sup_{x\in\R} |u(t,x;g\cdot s,u_0)-\phi(x-c(g\cdot s)t+\xi_0;g\cdot s)|\to 0 \, \hbox{ as } \,t\to+\infty,
	\end{equation*}
	for some constant $\xi_0\in\R$ (depending on $g$ and $s$).
\end{itemize}
\end{lem}

\begin{rem}\label{rem-expon}
We remark that statement {\rm (i)} actually holds in a more general situation. More precisely, for any traveling wave solution $\tilde{\phi}(x-\tilde{c})$ of \eqref{frozen-eq-g} with $\tilde{\phi}(\cdot)$ being decreasing in $\R$ and $\tilde{\phi}(+\infty)=0$ (resp. $\tilde{\phi}(-\infty)=1$), one has $\tilde{\phi}(\cdot) \sim A_1 \me^{-\tilde{\lambda}_1\xi}$ as $\xi\to+\infty$ (resp. $ 1- \tilde{\phi}(\cdot) \sim A_2 \me^{\tilde{\lambda}_2\xi}$), where  $\tilde{\lambda}_1$ and $\tilde{\lambda}_2$ are defined similarly as in \eqref{decay-rate} with $c(g\cdot s)$ changed by $\tilde{c}$ (see {\rm \cite{aw,fm}}).
\end{rem}

Our first main result in this subsection concerns the continuity of the functions $\phi(\cdot;g\cdot s)$ and $c(g\cdot s)$ with respect to $s\in\R$ and $g\in H(f)$. To do so, let us first show the uniform boundedness of the wave speed $c(g\cdot s)$.

\begin{lem}\label{speed-bound}
The set $\{c(g\cdot s): g\in H(f), s\in\R \}$ is bounded.
\end{lem}
\begin{proof}
Since $f$, $\partial_uf$ and $\partial_{uu}f$ are bounded in $s\in\R$ and $u$ in bounded sets, it follows from Lemma \ref{almost-partial} that $g$ and $\partial_u g$ are bounded in $s\in\R$, $u\in [-\delta_0,1+\delta_0]$ uniformly for $g\in H(f)$. This implies that there exists some $\gamma'_0>\gamma_0$ (independent of $g$ and $s$) such that
\begin{equation*}
	\left\{\baa{ll}
	g(s,u)\ge-\gamma'_0 u & \hbox{for all } \,  u\in [-\delta_0,\delta_0],\vspace{5pt}\\
	g(t,u)\le\gamma'_0(1-u) & \hbox{for all }\,  u\in [1-\delta_0,1+\delta_0].  \eaa\right.
\end{equation*}
Combining this with \eqref{gsn-delta}, one can find two smooth functions $g^{\pm}: \R \to \R$ which are of bistable type in the sense that
$$g^{\pm}(0)=g^{\pm}(1)=0,\quad g^{\pm}(u)<0 \, \hbox{ in } \,  (0,\theta^{\pm}),\quad g^{\pm}(u)>0  \,\hbox{ in } \, (\theta^{\pm}, 1),$$
for some $\delta_0< \theta^{+}\leq \theta^{-} < 1-\delta_0$, and satisfy $(g^{\pm})'(0)<0$, $(g^{\pm})'(1)<0$, and
\begin{equation}\label{auto-sub-super}
	g^-(u) \leq g(s,u)  \leq g^+(u)  \, \hbox{ for all }\, u\in [0,1],\,g\in H(f),\, s\in\R.
\end{equation}
By classical theory on traveling wave solutions of bistable equations \cite{fm}, the following two homogeneous equations
\begin{equation}\label{eq-auto-sub}
	v^{\pm}_t=v^{\pm}_{xx}+ g^{\pm}(v^{\pm}) \, \hbox{ for } \, t\in\R,\, x\in\R,
\end{equation}
admit a unique (up to shifts) traveling wave solution $v^{\pm}(t,x)=\phi^{\pm}(x-c_0^{\pm}t)$ satisfying $\phi^{\pm}(-\infty)=1$ and $\phi^{\pm}(+\infty)=0$.

Now, we claim that
\begin{equation*}%\label{claim-bound-c}
c^-_0\leq c(g \cdot s) \leq c^+_0 \, \hbox{ for all } \, g\in H(f),\, s\in\R.
\end{equation*}
We only show the first inequality in the above claim, as the proof of the second one is similar. Indeed, thanks to \eqref{auto-sub-super},
by the comparison principle for parabolic equations, for each $g\in H(f)$ and $s\in\R$, there holds
\begin{equation}\label{comp-v-phin}
v^-(t,x;\phi(\cdot;g \cdot s)) \leq \phi(x-c(g \cdot s)t;g\cdot s) \, \hbox{ for all } \, t\geq 0,\,x\in\R,
\end{equation}
where $v^-(t,x;\phi(\cdot;g\cdot s))$ is the solution of the Cauchy problem of \eqref{eq-auto-sub} with nonlinearity $g^-$ and initial function $v^-(0,x;\phi(\cdot;g \cdot s))=\phi(x;g \cdot s)$. On the other hand,
by the stability of bistable traveling waves (see Lemma \ref{basic-phi} (ii)), there exists some $\xi_*=\xi_*(g\cdot s)$ such that
$$|v^-(t,x;\phi(\cdot;g\cdot s))- \phi^-(x-c^-_0t+\xi_*)|\to 0\, \hbox{ as }\, t\to+\infty \,\hbox{ uniformly in } \,x\in\R. $$
With choosing $x=c(g\cdot s)t$, this together with \eqref{comp-v-phin} and \eqref{normal-phi-g} implies that
\begin{equation}\label{phi--phin}
 \lim_{t\to+\infty} \phi^-((c(g\cdot s)-c^-_0)t+\xi_*) \leq  \phi(0;g\cdot s) =\frac{1}{2} \, \hbox{ for all }\, g\in H(f),\, s\in\R.
\end{equation}
Since $\phi^-(-\infty)=1$, there must hold $c(g\cdot s)\geq c^-_0$. Proceeding similarly as above, one can prove that $c(g\cdot s)\leq c^+_0$. This completes the proof of Lemma \ref{speed-bound}.
\end{proof}

\begin{pro}\label{continuity}
For any sequences $(s_n)_{n\in\N}\subset \R$ and $(g_n)_{n\in\N}\subset H(f)$ such that
\begin{equation}\label{sequence-gn-sn}
g_n(s_n,u) \to g_*(s_*,u)\,\,\hbox{ uniformly in } \,\, u\in [0,1] \,\,\hbox{ as }\,\, n\to\infty
\end{equation}
for some $s_*\in\R$ and $g_*\in H(f)$, there holds
\begin{equation*}
	c(g_n\cdot s_n) \to c(g_*\cdot s_*)\,\,\hbox{ as }\,\, n\to\infty,
\end{equation*}
 and
\begin{equation*}
\phi(\cdot;g_n\cdot s_n) \to \phi(\cdot;g_*\cdot s_*)\,\,\hbox{ in }\,\,C^{2}_{loc}(\R)\,\,\hbox{ as } \,\, n\to\infty.
\end{equation*}
\end{pro}

\begin{proof}
Let us first  show the convergence of $c(g_n\cdot s_n)$ to $c(g_*\cdot s_*)$  as $n\to \infty$.  By Lemma \ref{speed-bound}, up to extraction of a subsequence, we may assume that $c(g_n\cdot s_n) \to c_{\infty}$ as $n\to \infty$ for some $c_{\infty}\in [c^-_0, c^+_0]$. To complete the proof, it suffices to show that $c_{\infty}= c(g_*\cdot s_*)$.
We first prove that $c_{\infty}\geq c(g_*\cdot s_*)$. For each $n\in\N$, let $y_n$ be the unique real number such that
\begin{equation}\label{normal-phi-new}
\phi(y_n;g_n\cdot s_n)=\delta_0.
\end{equation}
Writing $\psi_{n}(\xi)=\phi(\xi+y_n;g_n\cdot s_n)$, we see that $\psi_{n}(\xi)$ is decreasing in $\xi\in\R$ and
\begin{equation*}
	\left\{\baa{l}
	\psi_n''+c(g_n\cdot s_n)\psi_n'+g_n(s_n, \psi_n)=0   \,\,\hbox{ in }\R,\vspace{5pt}\\
	0< \psi_n < 1\hbox{ in }\R\quad\hbox{and}\quad \psi_n(0)=\delta_0.
	\eaa\right.
\end{equation*}
Notice that by Lemma \ref{almost-partial}, $\partial_ug_n(s,u)$ is bounded in $s\in\R$, $u\in [0,1]$ uniformly for $n\in\N$.
By using \eqref{sequence-gn-sn} and standard elliptic estimates, we find a $C^2(\R)$ nonincreasing function $0\leq \psi_{\infty}\leq 1$ such that, up to extraction of a further subsequence, $\psi_n\to \psi_{\infty}$ in $C^2_{loc}(\R)$ as $n\to+\infty$, and $\psi_{\infty}$ solves
\begin{equation}\label{eq-psi-infty}
\psi_{\infty}''+c_{\infty}\psi_{\infty}'+g_*(s_*,\psi_{\infty})=0\,\,\hbox{ in }\,\,\R,\quad \hbox{and} \quad \psi_{\infty}(0)=\delta_0.
\end{equation}
By the strong maximum principle, we have $0<\psi_{\infty}<1$ in $\R$. Since $\psi_{\infty}$ is nonincreasing, the limits $\psi_{\infty}(\pm\infty):=\lim_{\xi\to\pm\infty}\psi_{\infty}(\xi)$ exist and they are zeros of $g_*(s_*,\cdot)$. Since $g_*(s_*,\cdot)$ has no zeros in $(0,\delta_0)$ (due to \eqref{gsn-delta}) and since $\psi_{\infty}(0)=\delta_0$, we have $\psi_{\infty}(+\infty)=0$ and $\psi_{\infty}(-\infty)\in (\delta_0,1]$. In other words, $\psi_{\infty}(x-c_{\infty}t)$ is a traveling wave solution of the following equation
\begin{equation}\label{cauchy-g*}
	v_t=v_{xx}+g_*(s_*,v)\,\, \hbox{ for }\,  t\in\R, \,x\in\R,
\end{equation}
which connects $0$ and some zero of the function $g_*(s_*,\cdot)$ in $(\delta_0,1]$. Let $v_0\in C(\R)\cap L^{\infty}(\R)$ be a nonincreasing function such that $v_0(-\infty)\in (1,1+\delta_0)$, $v_0(+\infty)\in (0,\delta_0)$ and $v_0(\cdot) \geq \psi_{\infty}(\cdot)$ in $\R$, and let $v(t,x;v_0)$ be the solution of the Cauchy problem of \eqref{cauchy-g*} with initial function $v(t,x;v_0)=v_0(x)$ in $\R$. The comparison principle implies that $\psi_{\infty}(x-c_{\infty}t) \leq v(t,x;v_0)$ for all $t\geq 0$ and $x\in\R$. Then, similar arguments to those used in showing \eqref{phi--phin} (by applying Lemma \ref{basic-phi} (ii) to the wave $\phi(x-c(g_*\cdot s_*)t;g_*\cdot s_*)$) yield that
$$ \delta_0=\psi_{\infty}(0) \leq \lim_{t\to+\infty} \phi((c_{\infty}-c(g_*\cdot s_*))t+\xi_*;g_*\cdot s_*),$$
for some $\xi_*\in\R$. Since $\phi(+\infty;g_*\cdot s_*)=0$, we must have $c_{\infty}\leq c(g_*\cdot s_*)$.

In a similar way, one can prove that  $c_{\infty}\geq c(g_*\cdot s_*)$ by changing the normalization condition \eqref{normal-phi-new} into $\phi(z_n;g_n\cdot s_n)=1-\delta_0$ with $z_n\in\R$. Combining the above, we obtain that $c_{\infty}= c(g_*\cdot s_*)$. This implies that the whole sequence $(c(g_n\cdot s_n))_{n\in\N}$ converges to $c(g_*\cdot s_*)$ as $n\to\infty$.

Let us now turn to the proof of the convergence of $\phi(\cdot;g_n\cdot s_n)$ to $\phi(\cdot;g_*\cdot s_*)$ in $C^{2}_{loc}(\R)$ as $n\to \infty$. Remember that $\psi_{\infty}\in C^2(\R^2)$ is a decreasing solution of \eqref{eq-psi-infty}. We already know that $\psi_{\infty}(+\infty)=0$ and $\psi_{\infty}(-\infty)\in (\delta_0, 1]$. From the strong maximum principle applied to the equation satisfied by $\psi_{\infty}'$, one has $\psi'_{\infty}<0$ in $\R$. Moreover, since $c_{\infty}=c(g_*\cdot s_*)$, it follows from Lemma \ref{basic-phi} (i) and Remark \ref{rem-expon} that $\psi_{\infty}(\xi)$ and $\phi(\xi;g_*\cdot s_*)$ decay to $0$ as $\xi\to+\infty$ exponentially fast with the same decay rate. Namely, there holds
\begin{equation}\label{decay-phi*-psi}
	\phi(\xi;g_*\cdot s_*)\sim A_1 \me^{-\lambda_1(g_*\cdot s_*)\xi} \quad \hbox{and}\quad  \psi_{\infty}(\xi) \sim  \tilde{A}_1 \me^{-\lambda_1(g_*\cdot s_*)\xi} \quad \hbox{as }\, \xi\to+\infty,
\end{equation}
for some constants $A_1>0$ and $\tilde{A}_1>0$, where $\lambda_1(g_*\cdot s_*)>0$ is provided by \eqref{decay-rate}.

Next, we prove that $\psi_{\infty}(-\infty)=1$. Assume by contradiction that $\psi_{\infty}(-\infty)<1$. Then, since  $\phi(-\infty;g_*\cdot s_*)=1$ and since both $\psi_{\infty}(\xi)$ and $\phi(\xi;g_*\cdot s_*)$ are decreasing in $\xi\in\R$, it follows from \eqref {decay-phi*-psi} that $\psi_{\infty}(\xi) \leq \phi(\xi-l;g_*\cdot s_*)$ for all sufficiently large $l>0$. Then we can define
$$l_*:=\inf\{l\in\R: \psi_{\infty}(\xi) \leq \phi(\xi-l;g_*\cdot s_*) \,\hbox{ for all }\, \xi\in\R \}.$$
Clearly, $l_*$ is a real number, and $\psi_{\infty}(\cdot) \leq \phi(\cdot-l_*;g_*\cdot s_*)$ in $\R$. Furthermore, since  $\psi_{\infty}(-\infty)<\phi(-\infty;g_*\cdot s_*)$, it follows from the strong maximum principle for parabolic equations (recall that $\psi_{\infty}(x-c(g_*\cdot s_*)t)$ and $\phi(x-c(g_*\cdot s_*)t-l_*;g_*\cdot s_*)$ are entire solutions of \eqref{cauchy-g*}) that
$$\psi_{\infty}(\cdot) < \phi(\cdot-l_*;g_*\cdot s_*)\,\hbox{ in }\, \R.$$
Then, by \eqref{decay-phi*-psi} again, two possibilities may happen: (a) $\tilde{A}_1< A_1 \me^{\lambda_1(g_*\cdot s_*)l_*}$; (b) $\tilde{A}_1= A_1 \me^{\lambda_1(g_*\cdot s_*)l_*}$. If case (a) happens, then one can find some small $\epsilon>0$ such that $\psi_{\infty}(\cdot) \leq \phi(\cdot-l_*-\epsilon;g_*\cdot s_*)$ in $\R$, which is a contradiction with the definition of $l_*$.

Case (b) would lead to a contradiction as well. As a matter of fact, set $w(\xi):=\phi(\xi-l_*;g_*\cdot s_*)-\psi_{\infty}(\xi)$ for $\xi\in\R$. Clearly, $w(\xi)>0$ for $\xi\in\R$ and
$$w''(\xi)+c(g_*\cdot s_*)w'(\xi)+h(\xi)w(\xi)=0\, \hbox{ for }\,\xi\in\R,  $$
where $h\in L^{\infty}(\R)$ is given by
$$ h(\xi)=\frac{g_*(s_*,\phi(\xi-l_*;g_*\cdot s_*))-g_*(s_*,\psi_{\infty}(\xi))}{\phi(\xi-l_*;g_*\cdot s_*)-\psi_{\infty}(\xi)}. $$
If case (b) happens, then one has
\begin{equation}\label{expon-w}
w(\xi)=o(\me^{-\lambda_1(g_*\cdot s_*)\xi})\,\hbox{ as }\,\xi\to+\infty.
\end{equation}
On the other hand, since the function $\partial_u g_*(s_*,\cdot)$ is locally Lipschitz-continuous, it is easily seen from \eqref{decay-phi*-psi} that there exists some constant $A_2>0$ such that $|h(\xi)-\partial_ug_*(s_*,0)|\leq  A_2 \me^{-\lambda_1(g_*\cdot s_*)\xi}$ for all large $\xi>0$. Then, by the asymptotic stability theory for ordinary differential equations (see e.g., \cite[Chapter 13]{cl}), one can conclude that
$$w(\xi)\sim A_3 \me^{-\lambda_1(g_*\cdot s_*)\xi} \,\hbox{ as } \, \xi\to+\infty $$
for some constant $A_3>0$. This contradicts \eqref{expon-w}, and hence, both case (a) and case (b) cannot happen. Therefore,
we have  $\psi_{\infty}(-\infty)=1$.

As a consequence, $\psi_{\infty}(x-c(g_*\cdot s_*)t)$ is a traveling wave solution of \eqref{cauchy-g*} connecting $0$ and $1$. Then, by the uniqueness of such waves, one has $\psi_{\infty}(\cdot)\equiv \phi(\cdot+y_{\infty};g_*\cdot s_*)$, where $y_{\infty}$ is the unique point such that $\phi(y_{\infty};g_*\cdot s_*)=\delta_0$, and hence, from the above proof, one sees that the whole sequence $(\psi_{n}(\cdot))_{n\in\N}$ converges to $\phi(\cdot+y_{\infty};g_*\cdot s_*)$ in $C^2_{loc}(\R)$ as $n\to\infty$. By the normalization condition \eqref{normal-phi-g}, this implies that $y_n\to y_{\infty}$ and $\phi(\cdot;g_n\cdot s_n)\to \phi(\cdot;g_*\cdot s_*)$ in $C^2_{loc}(\R)$ as $n\to\infty$. The proof of Proposition \ref{continuity} is thus complete.
\end{proof}

By using the above continuity result, we show in the next proposition that, the wave profile $\phi(\xi;g\cdot s)$ approaches its limiting states as $\xi\to\pm\infty$ exponentially fast uniformly in $s\in\R$ and $g\in H(g)$.

\begin{pro}\label{uniform-decay}
	The following statements hold true:
	\begin{itemize}
		\item [{\rm (i)}] $\lim_{\xi\to+\infty} \phi(\xi;g\cdot s) =0$ and $\lim_{\xi\to -\infty} \phi(\xi;g\cdot s) =1$ uniformly in $s\in\R$ and $g\in H(f)$;
		
		\item [{\rm (ii)}] 	For any $\delta\in (0,1/2)$, there exists $\beta=\beta(\delta)>0$ $($independent of $s\in\R$ and  $g\in H(f)$$)$ such that
		\begin{equation*}
			\partial_{\xi} \phi(\xi;g\cdot s) \leq -\beta \,\hbox{ for all }\,\xi\in\R, \, s\in\R,\, g\in H(f)\, \hbox{ such that }\, \delta \leq \phi(\xi;g\cdot s) \leq 1-\delta;
		\end{equation*}
		
		 \item [{\rm (iii)}] There exist positive constants $C_1$, $C_2$ and $M$ $($all are independent of $s\in\R$ and $g\in H(f)$$)$ such that
		 \begin{equation}\label{psi-exp}
		 	\left\{\baa{ll}
		 	0< \phi(\xi;g\cdot s) \leq C_1 \me^{-\mu_1(g\cdot s) \xi}  & \hbox{for all  }\,\xi\geq M,\, s\in\R,\, g\in H(f), \vspace{5pt}\\
		 	0< 1-\phi(\xi; g\cdot s) \leq C_2 \me^{\mu_2 (g\cdot s) \xi}  & \hbox{for all  }\,\xi\leq -M,\, s\in\R,\, g\in H(f),
		 	\eaa\right.
		 \end{equation}
	  where $\mu_1(g\cdot t)$ and $\mu_2(g\cdot t)$ are positive constants given by
	  \begin{equation}\label{decay-rate-mu}
	 	\left\{\baa{l}
	 	\displaystyle   \mu_1(g\cdot s)=\frac{c(g\cdot s)+\sqrt{(c(g\cdot s))^2+4\gamma_0}}{2},\vspace{6pt}\\
	 	\displaystyle	\mu_2(g\cdot s)=\frac{-c(g\cdot s)+\sqrt{(c(g\cdot s))^2+4\gamma_0}}{2}.  \eaa\right.
	 \end{equation}
     with $\gamma_0>0$ provided by {\rm (A2)}.
	\end{itemize}	
\end{pro}

\begin{proof}
The proof is similar to that of \cite[Proposition 4.1 (ii)-(iii)]{dhl} in the periodic case. For completeness, we include the details as follows.

As for statement (i), we only show the convergence of $\phi(\xi;g\cdot s)$ as $\xi\to+\infty$, since the proof of the other one is identical. Assume by contradiction that there exist $\epsilon_1\in(0,1)$ and a sequence~$(\xi_n,s_n,g_n)_{n\in\N}\subset \R\times\R\times H(f)$ such that $\xi_n\to+\infty$ as $n\to+\infty$ and
$$\phi(\xi_n;g_n\cdot s_n) \geq \epsilon_1 \,\hbox{ for each } \,n\in\N. $$
Since the set $H(f)$ is compact under the compact open topology (see Section 2.1), up to extraction of some subsequence, one can find some $s_*\in\R$ and $g_*\in H(f)$ such that \eqref{sequence-gn-sn} holds. By Proposition \ref{continuity},
we have $\phi(\xi;g_n\cdot s_n)\to \phi(\xi;g_*\cdot s_*)$ locally uniformly in $\xi\in\R$ as $n\to\infty$.
Let $\xi_{\infty}\in\R$ be the unique point such that $\phi(\xi_{\infty};g_*\cdot s_*)=\epsilon_1/2$. Since $\xi_n\to+\infty$ as $n\to+\infty$, we have $\xi_n>\xi_{\infty}$ for all large $n$, whence by the monotonicity of $\phi(\cdot;g_n\cdot s_n)$, there holds $\phi(\xi_{\infty};g_n\cdot s_n)>\phi(\xi_n;g_n\cdot s_n)\ge \epsilon_1$. Passing to the limit as $n\to+\infty$, we obtain $\phi(\xi_{\infty}; g_*\cdot s_*) \geq \epsilon_1$, which is impossible. Therefore, $\phi(\xi;g\cdot s) \to 0$ as $\xi\to+\infty$ uniformly in $s\in\R$ and $g\in H(f)$.

Next, we show statement (ii). Assume by contradiction that there exist $\tilde{\delta}\in (0,1/2)$ and  a sequence~$(\tilde{\xi}_n,\tilde{s}_n,\tilde{g}_n)_{n\in\N}\subset \R\times\R\times H(f)$ such that
$\tilde{\delta}\leq  \phi(\tilde{\xi}_n;\tilde{g}_n\cdot \tilde{s}_n) \leq 1-\tilde{\delta}$ for each $n\in\N$, and that \begin{equation}\label{partial-phin}
\partial_{\xi}\phi(\tilde{\xi}_n;\tilde{g}_n\cdot \tilde{s}_n) \to 0\,  \hbox{ as }\,n\to\infty.
\end{equation}
By (i), the sequence $(\tilde{\xi}_n)$ is bounded, and hence, up to extraction of some subsequence, $\tilde{\xi}_n\to \tilde{\xi}_* \in \R$ as $n\to\infty$. Furthermore, it follows from Proposition \ref{continuity} that, up to extraction of a further subsequence, $\partial_{\xi} \phi(\tilde{\xi}_n;\tilde{g}_n\cdot \tilde{s}_n) \to \partial_{\xi}\phi(\tilde{\xi}_*;\tilde{g}_*\cdot \tilde{s}_*)$ as $n\to\infty$ for some $\tilde{g}_*\in H(f)$ and $\tilde{s}_*\in\R$. Since $\phi(\cdot;\tilde{g}_*\cdot \tilde{s}_*)$ is decreasing, this contradicts with \eqref{partial-phin}. Therefore, statement (ii) is proved.

It remains to show statement (iii). Let $g\in H(f)$ and $s\in\R$ be arbitrary. By (i), there exists $M>0$ (independent of $g$ and $s$) such that
\begin{equation*}
	\left\{\baa{ll}
	0<\phi(\xi;g\cdot s)\leq \delta_0 & \hbox{for all } \,  \xi\geq M,\vspace{5pt}\\
	1-\delta_0 \leq  \phi(\xi;g\cdot s) <1 & \hbox{for all }\, \xi\leq -M,  \eaa\right.
\end{equation*}
where $\delta_0\in (0,1/2)$ is the constant provided by (A2). Then, by the first line of \eqref{gsn-delta},  we have
$$\partial_{\xi\xi}\phi(\xi;g\cdot s) +c(g\cdot s) \partial_{\xi}\phi(\xi;g\cdot s)-\gamma_0 \phi(\xi;g\cdot s)\geq 0  \,\hbox{ for all }\, \xi\geq M. $$
On the other hand, thanks to Lemma \ref{speed-bound}, one can choose a real number $\mu_1^* = \sup_{g\in H(f),\,s\in\R}\mu_1(g\cdot s)>0$.
Then, we define $v_1(\xi)=C_1\me^{-\mu_1(g\cdot s)\xi}$ with $C_1=\delta_0 \me^{\mu_1^* M}$. It is straightforward to check that for all $s\in\R$ and $g\in H(f)$, there holds $\phi(M;g\cdot s)\leq v_1(M)$ and
$$v_1''(\xi)+c(g\cdot s) v_1'(\xi)-\gamma_0 v_1(\xi)\leq 0\, \hbox{ for all } \, \xi\geq M.$$
Consequently, the first line of \eqref{psi-exp} follows directly from the elliptic weak maximum principle. Similarly, by comparing $1-\phi(\xi;g\cdot s)$ and $v_2(\xi):=C_2 \me^{\mu_2\xi}$ over the interval $\xi\in (-\infty,-M]$, where $C_2=\delta_0 \me^{-\mu_2^* M}$ with $\mu_2^* = \sup_{g\in H(f),\,s\in\R}\mu_2(g\cdot s)>0$,
one can prove the second line of \eqref{psi-exp}.
\end{proof}

In the remaining part of this subsection, we fix $g\in H(f)$, and show the almost periodicity, as well as the $C^1$- smoothness, of the homogeneous wave $(\phi(\cdot;g\cdot s),c(g\cdot s))$ with respect to $s$. Without loss of generality, we prove these results for the case $g=f$. Clearly, $(\phi(\xi;f\cdot s), c(f\cdot s))$ satisfies
\begin{equation}\label{eq-fs}\left\{\baa{l}
	\partial_{\xi\xi} \phi(\xi;f\cdot s)+c(f\cdot s)\partial_{\xi} \phi(\xi;f\cdot s)+f(s, \phi(\xi;f\cdot s))=0
	\quad\hbox{for }\,\,\xi\in\R, \vspace{5pt}\\
	\phi(-\infty;f\cdot s)=1,\quad \phi(+\infty;f\cdot s)=0. \eaa\right.
\end{equation}

\begin{pro}\label{uniform-almost}
$c(f\cdot s)$ is almost periodic in $s$, and the functions $s\mapsto \phi(\xi;f\cdot s)$, $s\mapsto \partial_{\xi}\phi(\xi;f\cdot s)$ and $s\mapsto \partial_{\xi\xi}\phi(\xi;f\cdot s)$ are almost periodic uniformly with respect to $\xi\in\R$.
\end{pro}

\begin{proof}
Since the function $s\mapsto f(s,u)$ is almost periodic uniformly in $u\in [0,1]$, by Lemma \ref{eqdf-almost} (i),
for any sequences $(s'_n)\subset \R$, $(t'_n)\subset \R$,
there exist subsequences  $(s_n)\subset (s'_n)$, $(t_n)\subset (t'_n)$ such that for each $s\in\R$ and $u\in [0,1]$,
$$\lim_{m\to\infty} \lim_{n\to\infty} f(s+t_n+s_m,u) = \lim_{n\to\infty} f(s+t_n+s_n,u).$$
It further follows from  Lemma \ref{eqdf-almost} (ii) that, the above convergences indeed hold uniformly in $s\in\R$, $u\in [0,1]$. This particularly implies that for each $s\in\R$,
$$f(s+t_{n}+s_{n},u) \to f_*(s+s_*,u) \,\hbox{ as }\, n\to\infty \,\hbox{ uniformly in }\,u\in [0,1], $$
and
$$f(s+t_{n}+s_{m},u) \to f_*(s+s_*,u)\,\hbox{ as }\, n,m\to\infty \,\hbox{ uniformly in }\,u\in [0,1],$$
for some $f_*\in H(f)$ and $s_*\in\R$.
It then follows from Proposition \ref{continuity} that for each $s\in\R$,
$$\lim_{n\to\infty} c(f\cdot (s+t_{n}+s_{n}))=\lim_{m\to\infty}\lim_{n\to\infty} c(f\cdot (s+t_{n}+s_{m}))= c(f_*\cdot (s+s_*)),$$
and
$$\lim_{n\to\infty} \phi(\xi;f\cdot (s+t_{n}+s_{n}))=\lim_{m\to\infty}\lim_{n\to\infty} \phi(\xi;f\cdot (s+t_{n}+s_{m}))= \phi(\xi;f_*\cdot (s+s_*))$$
locally uniformly in $\xi\in\R$. By Lemma \ref{eqdf-almost} (i)-(ii), this implies that  $c(f\cdot s)$ is almost periodic in $s$, and $\phi(\xi;f\cdot s)$ is almost periodic in $s$ uniformly for $\xi$ in bounded sets. Furthermore,
since $\phi(\xi;f\cdot s)$ approaches its limiting states as $\xi\to\pm\infty$ uniformly in $s\in\R$ (see Proposition \ref{uniform-decay} (i)), by using Lemma \ref{eqdf-almost} (ii) again, one obtains that  $\phi(\xi;f\cdot s)$ is almost periodic in $s$ uniformly with respect to $\xi\in \R$.

Finally, thanks to Lemma \ref{speed-bound}, applying standard elliptic estimates to equation \eqref{eq-fs}, it follows that $\partial_{\xi}\phi(\xi;f\cdot s)$, $\partial_{\xi\xi}\phi(\xi;f\cdot s)$ and $\partial_{\xi\xi\xi}\phi(\xi;f\cdot s)$ are bounded for $\xi\in\R$ and $s\in\R$. Since $\phi(\xi;f\cdot s)$ is almost periodic in $s$ uniformly for $\xi\in \R$, similarly to the proof of
Lemma \ref{almost-partial}, one can conclude that $\partial_{\xi}\phi(\xi;f\cdot s)$ and $\partial_{\xi\xi}\phi(\xi;f\cdot s)$ are almost periodic in $s$ uniformly with respect to $\xi\in \R$. This ends the proof of Proposition \ref{uniform-almost}.
\end{proof}

\begin{rem}\label{rem-bound-l2}
By {\rm Proposition \ref{uniform-decay} (iii)} and standard elliptic estimates applied to \eqref{eq-fs}, one has
\begin{equation}\label{psi-exp-s}
	\left\{\baa{ll}
|\partial_{\xi}\phi(\xi;f\cdot s)|+|\partial_{\xi\xi}\phi(\xi;f\cdot s)| \leq  \tilde{C}_1 \me^{-\mu_1(f\cdot s) \xi}  & \hbox{for all }\,\xi\geq M,\,s\in\R, \vspace{5pt}\\
|\partial_{\xi}\phi(\xi;f\cdot s)|+|\partial_{\xi\xi}\phi(\xi;f\cdot s)| \leq \tilde{C}_2 \me^{\mu_2(f\cdot s) \xi}  & \hbox{for all }\,\xi\leq -M,\,s\in\R,
	\eaa\right.
\end{equation}
where $M>0$, $\tilde{C}_1>0$ and $\tilde{C}_2$ are positive constants independent of $s\in\R$, and $\mu_1(f\cdot s)$, $\mu_2(f\cdot s)$ are defined as in \eqref{decay-rate-mu}. Since $\mu_1(f\cdot s)$ and $\mu_2(f\cdot s)$ are bounded from below by positive constants independent of $s\in\R$ due to {\rm Lemma \ref{speed-bound}}, it then follows that there exists $\tilde{C}_3>0$ (independent of $s\in\R$) such that
\begin{equation}\label{phi-bound-1}
	\|1-\phi(\cdot;f\cdot s)\|_{L^2((-\infty,0])}+\|\phi(\cdot;f\cdot s)\|_{L^2([0,\infty))}+\|\partial_\xi \phi(\cdot;f\cdot s) \|_{H^2(\R)} \leq \tilde{C}_3,
\end{equation}
for all $s\in\R$. Furthermore, by \eqref{psi-exp-s}, {\rm Proposition \ref{uniform-decay} (ii)} and {\rm Lemma  \ref{speed-bound}} again, one can find some constants $\tilde{C}_4>0$ and $\tilde{C}_5>0$ (both are independent of $s\in\R$) such that
\begin{equation}\label{esti-phi*}
\tilde{C}_4 \leq 	\int_{\R} \me^{2c(f\cdot s) \xi } | \partial_{\xi}\phi(\xi;f\cdot s) |^2d\xi + \int_{\R} 	\me^{c(f\cdot s) \xi}  |\partial_{\xi}\phi(\xi;f\cdot s)|^2 d\xi \leq \tilde{C}_5
\end{equation}
for all $s\in\R$.
\end{rem}

Our last result of this subsection is stated as follows.

\begin{pro}\label{smooth-s}
The function $s\mapsto c(f\cdot s)$ is of class $C^1(\R)$ and $\sup_{s\in\R} |c'(f\cdot s)|<+\infty$.
The function $(\xi,s) \mapsto \phi(\xi;f\cdot s)$ is of class $C^{2;1}_{\xi;s}(\R^2)$, and satisfies
\begin{equation}\label{esti-p-phi}
		\sup_{\xi\in\R,\,s\in\R} \left| \partial_s \phi(\xi;f\cdot s)\right| <+\infty.
\end{equation}
\end{pro}

\begin{proof}
Based on the above preparations, this proposition can be proved by using the implicit function theorem, where the arguments are essentially the same as those of \cite[Proposition 4.1 (iv)]{dhl} for the periodic case.
Therefore, we only outline the proof here and omit the details.

First of all, for mathematical convenience, we modify $f$ in $\R\times (\R\setminus [0,1])$ as follows:
\begin{equation*}%\label{extend-f}
	\left\{\baa{ll}
	f(t,u)=\partial_u f(t,0)u & \hbox{for all }(t,u)\in\R\times (-\infty,0),\vspace{5pt}\\
	f(t,u)=\partial_u f(t,1)(u-1) & \hbox{for all }(t,u)\in\R\times (1,\infty).\eaa\right.
\end{equation*}
Clearly, this modification does not affect the properties of $(\phi(\xi;f\cdot s), c(f\cdot s))$, as $\phi(\cdot;f\cdot s)$ ranges in $(0,1)$. It is also easily seen that with this modification, $f$ is of class $C^1$ in $\R^2$ ($f$ may be no longer of class $C^2$), and $f(t,u)$, $\partial_t f(t,u)$ $\partial_uf(t,u)$ are globally Lipschitz-continuous in $u\in\R$ uniformly in $t\in\R$.
This together with \eqref{phi-bound-1} and the assumption that $f(t,0)=f(t,1)=0$ for all $t\in\R$ implies that
for all $t\in\R$, $s\in\R$
\begin{equation}\label{f-bound-p}
\| f(t,\phi(\cdot;f\cdot s)) \|_{L^{2}(\R)} +\| \partial_t f(t,\phi(\cdot;f\cdot s)) \|_{L^{2}(\R)}
+ \| \partial_u f(t, \phi(\cdot;f\cdot s)) \|_{L^{2}(\R)} \leq C_1 ,
\end{equation}
where $C_1>0$ is a constant independent of $t$ and $s$.

Remember that $c(f\cdot s)$ is bounded in $s\in\R$ by Lemma \ref{speed-bound}. In the sequel, we use $c_0>0$ to denote a bound, that is, $|c(f\cdot s)| < c_0$ for all $s\in\R$.  Next, we fix a real number $\beta>0$, and for any $c\in \R$ and $s\in\R$, we define a liner operator:
$$M_{c,s}(v)=v''+cv'-\beta v \, \hbox{ for }\, v\in H^2(\R). $$
It follows from \cite[Lemma 5.3]{dhl} that the operator $M_{c,s}: H^2(\R) \to L^2(\R)$ is invertible, and that
\begin{equation}\label{inver-M}
\|M_{c,s}^{-1}(g)\|_{H^2(\R)} \leq C_2 \|g\|_{L^2(\R)}\,\hbox{ for }\, g\in L^2(\R),
\end{equation}
where $C_2>0$ is independent of $s\in\R$ and $c \in [-c_0,c_0]$.
Furthermore,  for any sequences  $(g_n)_{n\in\N} \subset L^2(\R)$, $(c_n)_{n\in\N}\subset \R$, $(s_n)_{n\in\R}\subset \R$ such that $\|g_n- g\|_{L^2(\R)} \to 0$, $c_n\to c$ and $s_n\to s$ as $n\to+\infty$, there holds
\begin{equation}\label{M-contin}
M_{c_n,s_n}^{-1}(g_n) \to  M_{c,s}^{-1}(g)  \, \hbox{ in }\, H^2(\R)\, \hbox{ as }\, n\to\infty,
\end{equation}
and the above convergence holds uniformly with respect to $s\in\R$, $c\in [-c_0,c_0]$ and $g$ in bounded sets of $H^2(\R)$.

In addition to $\beta>0$, we also choose an arbitrary $s_0\in\R$. It is easily seen from \eqref{phi-bound-1} that for any $s\in\R$, $\phi(\cdot;f\cdot s)-\phi(\cdot;f\cdot s_0) \in H^2(\R)$. We will show below that the function $D:\R\to H^2(\R)\times \R$ defined by
$$D(s)=(\phi(\cdot;f\cdot s)-\phi(\cdot;f\cdot s_0),c(f\cdot s)) \,\hbox{ for }\,s\in\R, $$
is continuously Fr\'echet differential at $s=s_0$. To do so, for any $(v,c,s)\in L^2(\R)\times \R\times \R$, we define $G(v,c,s)=(G_1,G_2)(v,c,s)$ with
$$G_1(v,c,s)=v+M_{c,s}^{-1}(K(v,c,s))\quad\hbox{and}\quad G_2(v,c,s)=v(0),$$
where
$$K(v,c,s):\xi\mapsto K(v,c,y)(\xi):=\partial_{\xi\xi} \phi(\xi;f\cdot s_0)+c\partial_{\xi}\phi(\xi;f\cdot s_0)+\beta v(\xi)+ f(s,v(\xi)+\phi(\xi;f\cdot s_0)).$$
It is clear that
$$G(D(s),s)=G(\phi(\cdot;f\cdot s)-\phi(\cdot;f\cdot s_0),c(f\cdot s),s)=(0,0)\,\hbox{ for all }\, s\in\R.$$
Moreover, by using \eqref{f-bound-p}, \eqref{inver-M} and \eqref{M-contin}, one can check that the function $G$ maps $H^2(\R)\times\R\times\R$ into $H^2(\R)\times\R$, and it is continuously Fr\'echet differentiable. In particular, there holds
\begin{equation}\label{fre-partial-s}
	\left\{\baa{l}
	\partial_{s}G_1(v,c,s)(\tilde{s})=-\tilde{s}\,M_{c,s}^{-1}\Big\{ \partial_{\xi\xi} \big[M_{c,s}^{-1}(K(v,c,s))-\phi(\cdot;f\cdot s_0)\big]-\partial_tf(s,v+\phi(\cdot;f\cdot s_0))\Big\},\vspace{5pt}\\
\partial_{s}G_2(v,c,s)(\tilde{s})=0,  \eaa\right.
\end{equation}
for all $\tilde{s}\in\R$.
Furthermore, with some obvious modifications of the proof of \cite[Lemma 5.4]{dhl} (the estimates in \eqref{esti-phi*} are useful in this step), one can prove that at $s=s_0$, the operator $Q_{s_0}= \partial_{(v,c)}G(0,c(f\cdot s_0),s_0): H^2(\R)\times \R \to H^2(\R)\times\R$ is invertible, and there exists a constant $C_3>0$ independent of $s_0$ such that
\begin{equation}\label{esti-inverse}
	\|Q_{s_0}^{-1}(\tilde{g},\tilde{d})\|_{H^2(\R)\times\R} \leq C_3 \|(\tilde{g},\tilde{d}) \|_{H^2(\R)\times\R} \,\hbox{ for all }\, (\tilde{g},\tilde{d})\in H^2(\R)\times\R,
\end{equation}
where the space $H^2(\R)\times\R$ is endowed with norm $ \|(\tilde{g},\tilde{d}) \|_{H^2(\R)\times\R}=\|\tilde{g}\|_{H^2(\R)}+|\tilde{d}|$.

Finally, we are able to apply the implicit function theorem to the function $G: H^2(\R)\times\R\times\R \to H^2(\R)\times\R$, and conclude that the function $D:\R\to H^2(\R)\times \R$ is continuously Fr\'echet differential in a neighborhood of $s=s_0$. Denote by $A_{s_0} \in \mathcal{L}(\R,H^2(\R)\times \R)$ the derivative operator of
$D(s)$ at $s=s_0$. By \eqref{fre-partial-s}, one checks that
$$A_{s_0}(\tilde{s})=-Q^{-1}_{s_0}\left(\partial_s G(0,c(f\cdot s_0),s_0)(\tilde{s})\right)=-\tilde{s}\,Q^{-1}_{s_0}\left(M^{-1}_{c(f\cdot s_0),s_0}\left(\partial_{\xi\xi}\phi(\cdot;s_0)+\partial_tf(s_0,\phi(\cdot;s_0))\right),0\right)$$
for all $\tilde{s}\in\R$. By a slight abuse of notation, we identify $A_{s_0}$ to an element in $H^2(\R)\times \R$, and then, it follows from \eqref{phi-bound-1}, \eqref{inver-M} and \eqref{esti-inverse} that
$$	\|A_{s_0}\|_{H^2(\R)\times\R} \leq C_4 $$
for some $C_4>0$ independent of $s_0\in\R$. Writing $A_{s_0}=(A_{s_0}^1,A_{s_0}^2) \in H^2(\R)\times\R$ for each $s_0\in\R$, one sees that $\partial_{s} \phi(\cdot;f\cdot s_0)=A_{s_0}^1$ and $c'(s_0)=A_{s_0}^2$, and hence,   $|c'(s_0)|\leq C_5$ and $\|\partial_{s} \phi(\cdot;f\cdot s_0)\|_{H^2(\R)}\leq C_5$. Since $s_0\in\R$ is arbitrary, the former immediately implies the $C^1$-smoothness of $s\mapsto c(f\cdot s)$; by the Sobolev inequality, the latter gives \eqref{esti-p-phi}. Furthermore, using standard elliptic estimates to the equation of $\partial_{s} \phi(\cdot;f\cdot s)$, one obtains that $(\xi,s) \mapsto \phi(\xi;f\cdot s)$ is of class $C^{2;1}_{\xi;s}(\R^2)$. The proof of Proposition \ref{smooth-s} is thus complete.
\end{proof}

%%%%%%%%%%%%%%%%%%%%%%%%%%%%%%%%%%%%%%%%
\subsection{Proof of Theorem \ref{th-slow}}

In addition to (A1), (A2), (A5), we assume that (A3) holds in this subsection. We will use the family of homogeneous waves $(\phi(\xi;f\cdot s),c(f\cdot s))_{s\in\R}$ to construct a pair of super-solution and sub-solution for
the following Cauchy problem 
\begin{equation}\label{cauchy-slow}
	\left\{\baa{ll}
	\displaystyle \partial_t w_T= \partial_{xx}w_T+ f\left(\frac{t}{T},w_T\right) & \hbox{for }\,\,t>0, \, x\in\R, \vspace{5pt}\\
	w_T(0,x)=\phi(x; f\cdot 0) & \hbox{for } \, x\in\R,
	\eaa \right.
\end{equation}
when $T$ is sufficiently large. As the arguments in this subsection does not involve with functions in $H(f)$, for convenience, we use $(\phi(\xi;s),c(s))$ to denote $(\phi(\xi;f\cdot s),c(f\cdot s))$ in the sequel. 

For each $T>0$, we define
\begin{equation*}%\label{de-XT}
	X_T(t)=\int_0^t c\left(\frac{\tau}{T}\right)d\tau\,\hbox{ for all }\, t\in\R.
\end{equation*}
Clearly, $X_T(0)=0$, and since $c(\cdot)$ is of class $C^1(\R)$ by Proposition \ref{smooth-s},  the function $t\mapsto X_T(t)$ is at least of class $C^2(\R)$. Furthermore, since the function $s\mapsto c(s)$ is almost periodic by Proposition \ref{uniform-almost}, it is easily seen that 
\begin{equation}\label{xT-t-lim}
\lim_{t\to \infty} \frac{X_T(t) }{t} =\lim_{t\to\infty} \frac{T}{t} \int_0^{t/T} c(s) ds =c_*.
\end{equation}

Let $\delta_1\in (0,\delta_0)$, $\gamma_1\in (0,\gamma_0)$ and $K_1>0$ be the constants provided by
\eqref{choose-gamma1} and \eqref{lip-cons}. The following lemma gives a super-solution of problem \eqref{cauchy-slow}. 

\begin{lem}\label{sup-slow}
	There exist $\epsilon_2\in (0,\delta_1/2)$ and $T_2=T_2(\epsilon_2)>0$ $($sufficiently large$)$ such that for any $T\in (T_2,+\infty)$, the function $w^+_T:[0,+\infty)\times\R$ defined by
	\begin{equation*}
		w^+_T(t,x)=\phi\left(x-X_T(t)+\kappa_T(t);\frac{t}{T}\right)+p_T(t)
	\end{equation*}	
	is a super-solution of \eqref{cauchy-slow} for $t\ge 0$ and $x\in\R$, where $p_T(\cdot)$ and $\kappa_T(\cdot)$ are $C^1([0,+\infty))$ functions satisfying
	\begin{equation}\label{require-pt}
		p_T(0)=\epsilon_2,\quad p_T'(t)<0< p_T(t) \, \hbox{ for all }\, t\geq 0,	
	\end{equation}
	and
	\begin{equation}\label{require-kappat}
		\kappa_T(0)=0,\quad  \kappa_T'(t)<0 \, \hbox{ for all }\, t\geq 0.	
	\end{equation}
\end{lem}

\begin{proof}
	By Proposition \ref{uniform-decay} (i), there exists $C_1>0$ such that for all $(\xi,s)\in \R\times\R$,
	\begin{equation}\label{slow-C1}
		\left\{\baa{ll}
		\displaystyle 0< \phi(\xi;s)\leq \delta_1/2 & \hbox{ if } \,\,\xi\geq C_1,\vspace{5pt}\\
		\displaystyle 1-\delta_1/2 \leq \phi(\xi;s)<1 & \hbox{ if }\,\,\xi\leq -C_1.
		\eaa \right.
	\end{equation}
    Since the function $\phi:\R^2\to (0,1)$ is continuous (see Proposition \ref{continuity}), and decreasing in the first variable and almost periodic with respect to the second variable (see Proposition \ref{uniform-almost}), it further follows that there exists $\epsilon_2\in (0,\delta_1/2] \subset (0,1/4)$ such that
    \begin{equation*}
     2\epsilon_2 \leq  \phi(\xi;s) \leq 1-2\epsilon_2\,\hbox{ for all }  -C_1\leq \xi \leq C_1,\, s\in\R.
    \end{equation*}
    Similarly to the strategy of the proof of Lemma \ref{sup-rapid}, to prove $w^+_T(t,x)$ is a super-solution,
    we will find $C^1([0,+\infty))$ functions $p_T$ and $\kappa_T$  satisfying \eqref{require-pt} and \eqref{require-kappat} respectively, such that if $T>0$ is sufficiently large, there holds
	\begin{equation*}
		N(t,x):= \partial_t w^+_T(t,x)-\partial_{xx} w^+_T(t,x) -f\left( \frac{t}{T}, w^+_T(t,x) \right) \geq 0
	\end{equation*}
	for all $(t,x)\in [0,+\infty)\times\R$. Since for each $s\in\R$, $(\phi(\xi;s), c(s))$ is a solution of
	\eqref{eq-fs}, it is straightforward to check that for any $T>0$ and any $C^1([0,+\infty))$ functions $q_T$ and $\eta_T$, one has
	$$	N(t,x)=
	\kappa_T'(t) \partial_{\xi}\phi\left(\xi;\frac{t}{T}\right)+p_T'(t)+\frac{1}{T}\partial_{s}\phi\left(\xi;\frac{t}{T}\right)+f\left(\frac{t}{T},\phi\right)-f\left(\frac{t}{T},\phi+p_T\right) $$
	for all $(t,x)\in [0,\infty)\times\R$, where $\xi=x-X_T(t)+\kappa_T(t)$, and $\partial_{\xi} \phi$, $\partial_{s} \phi$  stand for the partial derivatives of the function $\phi$ with respect to the first variable and the second variable, respectively. Thanks to Proposition \ref{smooth-s}, there is a constant $C_2>0$ such that for all $T>0$ and $(t,x)\in [0,+\infty)$,
	\begin{equation}\label{slow-ntx}	
		N(t,x)\geq   -\frac{C_2}{T}+
		\kappa_T'(t) \partial_{\xi}\phi\left(\xi;\frac{t}{T}\right)+p_T'(t)+f\left(\frac{t}{T},\phi\right)-f\left(\frac{t}{T},\phi+p_T\right).
	\end{equation}

    Let us first find a suitable function $p_T$ satisfying \eqref{require-pt} such that $N(t,x)\geq 0$ for $(t,x)\in [0,+\infty)$ such that $x-X_T(t)+\kappa_T(t)\geq C_1$. Indeed, notice that $p_T$ is required to satisfy $0<p_T(t)\leq \epsilon_2$ for all $t\geq 0$. It follows from  \eqref{slow-C1} that $0<\phi(\xi;t/T) \leq \delta_1/2$ and $0<\phi(\xi;t/T)+q_T(t) \leq \epsilon_2+\delta_1/2\leq \delta_1$ for $(t,x)\in [0,+\infty)$ with $\xi=x-X_T(t)+\kappa_T(t)\geq C_1$, whence by \eqref{choose-gamma1}, there holds
	$$f\left(\frac{t}{T},\phi\right)-f\left(\frac{t}{T},\phi+p_T\right)   \geq \gamma_1p_T(t).$$
	Since  $\kappa_T(t)$ is required to be decreasing in $t$ and since $\phi$ is decreasing in its first variable, it then follows from \eqref{slow-ntx} that
	$$N(t,x) \geq  -\frac{C_2}{T}+  p_T'(t)+\gamma_1p_T(t).$$
    We choose
    \begin{equation}\label{function-pt}
    	p_T(t)=\frac{C_2}{T\gamma_1}+\left(\epsilon_2-\frac{C_2}{T\gamma_1}\right) \me^{-\gamma_1t}\, \hbox{ for } \, t\geq 0.
    \end{equation}
    It is clear that
    \begin{equation*}
    	p_T'(t)+\gamma_1p_T(t)-\frac{C_2}{T}=0 \,\hbox{ for }\, t> 0,\quad\hbox{and}\quad  p_T(0)=\epsilon_2,
    \end{equation*}
    and that for each $T>T_2$, $p_T(t)$ satisfies \eqref{require-pt},  where $T_2>0$ is given by
	\begin{equation*}%\label{choose-T2}
		T_2=\frac{C_2 }{\epsilon_2 \gamma_1}.
	\end{equation*}
    Therefore, we obtain that for any $T>T_2$,  $N(t,x)\geq 0$ for all $[t,x)\in (0,\infty)\times\R$ with $x-X_T(t)+\kappa_T(t)\geq C_1$, provided that $\kappa_T$ satisfies \eqref{require-kappat}.
	
	In a similar way, with $p_T(t)$ given by \eqref{function-pt} and $\kappa_T$ required to satisfy \eqref{require-kappat}, one can prove that  for any $T>T_2$,  $N(t,x)\geq 0$ for all $[t,x)\in (0,\infty)\times\R$ with $x-X_T(t)+\kappa_T(t)\leq -C_1$.

	Finally, we find a suitable function $\kappa_T$ such that $N(t,x)\geq 0$ in the remaining region $(t,x)\in [0,\infty)\times\R$ such that $-C_1\leq x-X_T(t)+\kappa_T(t)\leq C_1$. In this case, by \eqref{slow-C1}, we have $2\epsilon_2 \leq  \phi(\xi;t/T) \leq 1-2\epsilon_2$ and $2\epsilon_2 \leq  \phi(\xi;t/T)+p_T(t)\leq 1-\epsilon_2$. It then follows from Proposition \ref{uniform-decay} (ii) that there exists $\beta_1>0$ (independent of $T$, $t$ and $x$) such that $\partial_{\xi} \phi\left(\xi; t/T \right) \leq -\beta_1$.
	Moreover, due to \eqref{lip-cons}, one has
	$$\left|f\left(\frac{t}{T},\phi\right)-f\left(\frac{t}{T},\phi+p_T\right)\right| \leq K_1p_T(t). $$
	Therefore, with $\kappa_T$ required to satisfy \eqref{require-kappat}, it follows from \eqref{slow-ntx} that
	$$N(t,x)\geq  -\frac{C_2}{T}-\beta_1\kappa_T'(t)+p_T'(t)-K_1p_T(t).   $$	
	With $p_T(t)$ given by \eqref{function-qt}, we then choose $\kappa_T(t)$ such that
	$$ -\frac{C_2}{T}-\beta_1\kappa_T'(t)+p_T'(t)-K_1p_T(t) =0,\quad\hbox{and}\quad \eta_T(0)=0.$$	
	Namely, we have
	\begin{equation}\label{function-kappa}
		\kappa_T(t)=-\frac{\gamma_1+K_1}{\beta_1\gamma_1}\left(\epsilon_2-\frac{C_2}{\gamma_1T}\right)\left(1-\me^{-\gamma_1t}\right)-\frac{(\gamma_1+K_1)C_2}{\beta_1\gamma_1T}t \, \hbox{ for }\,\, t\geq 0.
	\end{equation}
    It is then easily checked that for any $T>T_2$, $\kappa_T(t)$ satisfies \eqref{require-kappat} and that $N(t,x)\geq 0$ for all  $(t,x)\in [0,\infty)\times\R$ such that $-C_1\leq x-X_T(t)+\kappa_T(t)\leq C_1$.
	
	Combining the above, one can conclude that for any $T>T_2$, $w^+_T(t,x)$ is a super-solution of \eqref{cauchy-slow} for $t\ge 0$ and $x\in\R$. This completes the proof of Lemma \ref{sup-slow}.
\end{proof}

Similarly to  Lemma \ref{sup-slow}, we have the following sub-solution for problem \eqref{cauchy-slow}.

\begin{lem}\label{sub-slow}
	Let $\epsilon_2\in (0,\delta_1/2)$ and $T_2=T_2(\epsilon_2)>0$ be provided by Lemma~$\ref{sup-rapid}$. Then, for every $T>T_2$, the function $w^-_{T}:[0,+\infty)\times\R\to\R$ defined by
	$$w^-_{T}(t,x)=\phi\left(x-X_T(t)-\kappa_T(t);\frac{t}{T}\right)-p_T(t)$$
	is a sub-solution of \eqref{cauchy-slow} for $t\ge 0$ and $x\in\R$, where $p_{T}$ and $\kappa_{T}$ are $C^1([0,+\infty))$ functions given by \eqref{function-pt} and \eqref{function-kappa}, respectively.
\end{lem}

We are now ready to give the

\begin{proof}[Proof of Theorem $\ref{th-slow}$]
The proof follows the main lines as those of Theorem \ref{restate-th}. For clarity, we give its outline here.
Let $\epsilon_2\in (0,\delta_1/2)$ and $T_2=T_2(\epsilon_2)>0$ be as in Lemmas \ref{sup-slow}-\ref{sub-slow}, and let $T>T_2$ be arbitrary. By the comparison principle for parabolic equations, we have
\begin{equation*}%\label{sup-sub-wt}
	w_T^-(t,x)\leq  w_T(t,x)\leq w_T^+(t,x)\,\hbox{ for all }\, t\ge 0,\, x\in\R,
\end{equation*}
where $w_T(t,x)$ is the solution of the Cauchy problem \eqref{cauchy-slow}, and $w_T^-(t,x)$, $w_T^+(t,x)$ are, respectively, the sub- and super-solution provided by Lemmas \ref{sup-slow}-\ref{sub-slow}. On the other hand,
since the almost periodic traveling wave solution $U_T(t,x-c_T(t))$ is globally stable (see Theorem \ref{th-known} (ii)), it follows that
\begin{equation*}
	\sup_{x\in\R} \left|w_T(t,x)-U_T(t,x-c_T(t)-x^*_T)\right|\to 0 \, \hbox{ as }\, t\to+\infty,
\end{equation*}
for some constant $x^*_T\in\R$. Combining the above and setting $x:=x(t)=c_T(t)+x_T^*$, one finds some large $\tau_0>0$ such that
$$w_T^-(t,c_T(t)+x_T^*)-\frac{\epsilon_2}{4}\leq U_T(t,0) \leq w_T^+(t,c_T(t)+x_T^*)+\frac{\epsilon_2}{4} \, \hbox{ for all }\, t\geq \tau_0. $$
Then, by the normalization condition \eqref{normal}, and the fact $0<p_T(t)\leq p_T(0)=\epsilon_2$ for $t\geq 0$ (due to \eqref{require-pt}), this implies that
\begin{equation*}
	\phi\left(c_T(t)+x_T^*-X_T(t)-\kappa_T(t);\frac{t}{T}\right)-\frac{5}{4}\epsilon_2 \leq \frac{1}{2} \leq \phi\left(c_T(t)+x_T^*-X_T(t)+\kappa_T(t);\frac{t}{T}\right)+\frac{5}{4}\epsilon_2
\end{equation*}
for all $t\geq \tau_0$. Since $0<\epsilon_2<1/4$ and since $\lim_{x\to+\infty}\phi(x;s)=0$ and $\lim_{x\to-\infty}\phi(x;s)=1$ uniformly in $s\in\R$ by Proposition \ref{uniform-decay}, it follows that there exists a constant $z_0>0$ independent of $T$ such that
$$\left|c_T(t)+x_T^*-X_T(t)\right| \leq |\kappa_T(t)|+z_0 \,\hbox{ for all }\,\, t\geq \tau_0.  $$
Multiplying both sides of the above inequality by $1/t$ yields that
\begin{equation}\label{slow-ct-cs}
	\left|\frac{1}{t}\int_0^t c'_T(\tau)d\tau-\frac{T}{t}\int_0^{t/T} c(\tau)d\tau \right| \leq \frac{|x_T^*|+|c_T(0)|+z_0}{t}+\frac{|\kappa_T(t)|}{t}\,\hbox{ for all }\, t\geq \tau_0.
\end{equation}
Furthermore, it is easily seen from \eqref{function-kappa} that
$$\lim_{t\to+\infty}\frac{|\kappa_T(t)|}{t}= \frac{(\gamma_1+K_1)C_2}{\beta_1\gamma_1T}.$$
Therefore, setting $M_2=(\gamma_1+K_1)C_2/(\beta_1\gamma_1) $ and passing to the limits as $t\to+\infty$ in both sides of \eqref{slow-ct-cs}, one obtains from \eqref{xT-t-lim} that $|\bar{c}_T-c_*|\leq M_2/T$. Since $T>T_2$ is arbitrary, this completes the proof of Theorem \ref{th-slow}.
\end{proof}

\noindent{\bf Acknowledgements} This work has received funding from NSFC and the Basic and Applied Basic Research Foundation of Guangdong Province.

%\section*{Declarations}

%{\bf Conflict of interest} The author states that there is no conflict of interest.

%\vskip 0.2cm 

%\noindent{\bf Data availability} Data sharing is not applicable to this article as no datasets were generated or analyzed during the current study.

%%%%%%%%%%%%%%%%%%%%%%%%%%%%%%%%%%%%%%%%
%%%%%%%%%%%%%%%%%%%%%%%%%%%%%%%%%%%%%%%%

\end{document}